\documentclass[11pt]{article}
\usepackage{amsfonts}
\usepackage{amsthm}
\usepackage{amsmath}
\usepackage{amscd}
\usepackage{amsmath, amssymb}
\usepackage{float}

\usepackage[all]{xy}
\usepackage{tikz}
\usetikzlibrary{positioning} 
\usetikzlibrary{arrows,decorations.pathmorphing,shapes} 
\usepackage{tikz-cd}

\usepackage{graphicx, color}
\usepackage[T1]{fontenc}
\usepackage[math]{anttor}
\usepackage{caption}
\usepackage[latin2]{inputenc}
\usepackage{t1enc}
\usepackage[mathscr]{eucal}
\usepackage{indentfirst}
\usepackage{graphicx, color}
\usepackage{graphics}
\usepackage{pict2e}
\usepackage{epic}
\numberwithin{equation}{section}
\usepackage[margin=2.9cm]{geometry}
\usepackage{epstopdf} 
\usepackage{latexsym}
\usepackage{amsmath,amssymb,amsfonts}
\usepackage{graphicx} 
\usepackage{float}
\usepackage{amsmath, amsfonts, amssymb, amsxtra, amsthm, mathrsfs, amscd}
\usepackage{enumerate}

\usepackage{enumitem}
\graphicspath{{/home/tina/Desktop/tree}}

\graphicspath{{/home/tina/Desktop/tree}}

\newcommand{\Cg}{\mathcal{C}_g}

 \newcommand{\sys}{\operatorname{sys}}

\theoremstyle{plain}
\newtheorem{Thm}{Theorem}[section]
\newtheorem{Lem}[Thm]{Lemma}

\newtheorem{Prop}[Thm]{Proposition}

 \theoremstyle{definition}

\newtheorem{?}[Thm]{Problem}

\newtheorem*{Clm}{Claim}
\newenvironment{Clm*}
 {\pushQED{\qed}\Clm}
 {\popQED\endClm}

\newtheorem{Remark}{Remark}

\begin{document}

\title{Entropy and self-intersection number of geodesic currents on compact hyperbolic surfaces}




\date{}

 \maketitle
\centerline{ \author {Tina Torkaman}}

\begin{abstract}
Let $X$ be a compact hyperbolic surface of genus $g$, and $C$ a geodesic current on $X$. 
Denote by $h_X(C)$ the measure-theoretic entropy of $C$ with respect to the geodesic flow. Assume that $C$ is ergodic. In this paper, we establish a quantitative upper bound on $h_X(C)$ in terms of its self-intersection number $i(C,C)$ and the systole of $X$. In particular, we show that small self-intersection number forces small entropy.  
\end{abstract}

 \tableofcontents
\thispagestyle{empty}
\pagenumbering{arabic}

 \section{Introduction}
 Let $X$ be a compact hyperbolic surface of genus $g$. In this paper, for any ergodic geodesic current, we give an upper bound on its measure-theoretic entropy in terms of its self-intersection number and the systole of $X$. In the proof, we give an upper bound on the number of closed geodesics with a small self-intersection number compared to their length squared, which is interesting in its own right.
 
 \paragraph{Geodesic currents.}  Let $T_1(X)$ be the unit tangent bundle of $X$. A \textit{geodesic current} $C$ of $X$ is a Borel measure on $T_1(X)$ which is invariant under both geodesic flow and involution map $v \mapsto -v$ for $v \in T_1(X)$. Let $\Cg$ denote the space of all geodesic currents on a compact genus $g$ surface. Closed geodesics give examples of geodesic currents. Moreover, the set of weighted closed geodesics is dense in $\Cg$ equipped with weak topology. See \S \ref{sec: geod} for more details on geodesic currents.

 Bonahon showed that the notion of hyperbolic length $\ell_X$ and intersection number $i(-,-)$ between closed geodesics can continuously extend to $\Cg$. In particular, the length of $C \in \Cg$ is $\ell_X(C):=C(T_1(X))$, the total volume of $T_1(X)$ with respect to $C$.
 
 Given $C \in \Cg$ with $\ell_X(C)=1$, denote its measure-theoretic entropy with respect to the geodesic flow by $h_X(C)$. It is known that $h_X(C)=0$ when $i(C,C)=0$, that is, when $C$ is a \textit{measured lamination}. 
This follows from the work of Birman and Series~\cite{BirSer}, who showed that the Hausdorff dimension of the support of such a current is equal to $1$. 
Since the Hausdorff dimension of the support of $C$ is $\geq 2h_X(C)+1$
(see~\cite[Theorem~4.1]{fathi-entropy}), it follows that $h_X(C)=0$.

This naturally leads to the following question:

\textit{Does small self-intersection number $i(C,C)$ imply small entropy $h_X(C)$?}

 Our main result establishes the following inequality which gives an affirmative answer to this question when $C$ is ergodic. Let $\sys(X)$ denote the systole of $X$, which is the length of the shortest closed geodesic on $X$.

\begin{Thm}\label{thm: main.entropy}
 Let $C$ be an ergodic geodesic current with $\ell_X(C)=1$ and $i(C,C)\neq 0$. Then we have:
$$
h_X(C) \leq b_g \sqrt{i(C,C)}|\log(\frac{b_X}{\sqrt{i(C,C)}})|,
$$
where 
$$b_g\leq 385(g-1)\, , \, \, \, \, \, \, b_X \leq \frac{c_0(g-1)^2}{\sys(X)^2}, \, \, \, \, \, \, c_0\leq 5974.
$$
\end{Thm}

For more information on $b_X$, see Remarks \ref{remark: b}, \ref{remark: bc.bound}, \ref{remark: b1.1,3}. The argument in the proof of Theorem~\ref{thm: main.entropy} shows that $h_X(C)=0$ when $i(C,C)=0$.

The proof of Theorem \ref{thm: main.entropy} is organized into two main steps, which we explain below.

\paragraph{Step $1$: self-intersection number of random geodesics.} Given a geodesic current $C$ with $\ell_X(C)=1$, Katok showed, in a more general setting for negatively curved surfaces, $h_X(C)$ is bounded above by the growth rate of the number of closed geodesics on $X$ \cite{Ktok.entrpy}. More precisely, define $P_{\infty}(T)$ as the set of closed geodesics on $X$ with length $\leq T$. Katok showed 
 $$
 h_X(C) \leq \lim \limits_{T \to \infty} \frac{\log P_{\infty}(T)}{T}.
 $$
When $X$ is a hyperbolic surface, this growth rate is equal to $1$, and this upper bound is attained only for $C=L_X$, the Liouvile measure of $X$.

Let $\mathcal{G}$ denote the set of all closed geodesics on $X$. For $\epsilon \geq 0$, define

$$
P_{\epsilon}(T):=| \{\gamma \in \mathcal{G}: \ell_X(\gamma)\leq T, \, i(\gamma,\gamma)\leq \epsilon T^2 \}|.
$$

Inspired by Katok's method and by studying the self-intersection number of random geodesics with respect to $C$,  we obtain the following upper bound.

\begin{Thm}\label{thm: entropy<P}
Let $C$ be an ergodic geodesic current with $\ell_X(C)=1$. Then we have
$$
h_X(C)\leq \lim \limits_{T\to \infty} \frac{\log P_{i(C,C)}(T)}{T}.
$$
\end{Thm}
In the proof of Theorem \ref{thm: entropy<P}, we show that for a random closed geodesic $\gamma$, with respect to the current $C$, the normalized self-intersection number
\[
\frac{i(\gamma,\gamma)}{\ell_X(\gamma)^2}
\]
is close to $i(C,C)$; see Proposition~\ref{prop: entropy.a.e.self-int} and Lemma~\ref{lem: set=1}. 

For further details of the proof of Theorem~\ref{thm: entropy<P}, see \S\ref{sec: entropy}.

\paragraph{Step $2$: counting closed geodesics.} In this step, we obtain an upper bound on $P_{\epsilon}(T)$.

\begin{Thm}\label{thm: P-closed-geod}
For $\epsilon>0$, and $T$ large enough (in terms of $\epsilon$ and $g$) we have
$$
 P_{\epsilon}(T) \leq (\frac{b_X}{\sqrt{\epsilon}})^{b_g\sqrt{\epsilon}T}.
$$
where 
$$
b_g\leq 385(g-1) \, \, \, \, \, \, \, b_X \leq \frac{c_0(g-1)^2}{\sys(X)^2}\, , \, \, \, \, \, \, \,  c_0 \leq 5974.
$$
\end{Thm}

For more details on $b_X$, see Remarks \ref{remark: b}, \ref{remark: bc.bound}, \ref{remark: b1.1,3}.

\textbf{The idea behind the proof of Theorem \ref{thm: P-closed-geod}}. Fix a \emph{proper hexagon decomposition} of $X$ (see \S\ref{sec: count} for the definition). 
Given an oriented closed geodesic $\gamma$, we construct a closed curve $\phi(\gamma)$ that is homotopic to $\gamma$ as follows. The curve $\phi(\gamma)$ is obtained by homotoping $\gamma$ along the edges of the hexagons, such that points lying on an edge are moved only along that same edge throughout the homotopy. Then $\phi(\gamma)$ is made up of geodesic arcs within each hexagon. 
The construction ensures that every two oriented subarcs of $\phi(\gamma)$ whose start and end points lie on the same pair of edges do not intersect (see Figure~\ref{fig: intro.hexagon}). This non-intersection property is the key for our counting argument. 
Indeed, once we specify for each edge $e$ of a hexagon the edges to which the outgoing (or incoming) arcs of $\phi(\gamma)$ from $e$ are connected, then the curve $\phi(\gamma)$ is uniquely determined. 
Consequently, counting closed geodesics reduces to counting the possible symbolic words associated with the outgoing (or incoming) arcs at each edge. 
See \S\ref{sec: count} for more details.

\begin{figure}[h]
    \centering

\tikzset{every picture/.style={line width=0.75pt}} 

\begin{tikzpicture}[x=0.75pt,y=0.75pt,yscale=-0.7,xscale=0.7]

\draw   (100,171.4) -- (129.29,108) -- (226.91,108) -- (256.2,171.4) -- (226.91,234.8) -- (129.29,234.8) -- cycle ;
\draw  [dash pattern={on 4.5pt off 4.5pt}]  (122,124) -- (145.65,233.55) ;
\draw [shift={(146.08,235.51)}, rotate = 257.82] [color={rgb, 255:red, 0; green, 0; blue, 0 }  ][line width=0.75]    (10.93,-3.29) .. controls (6.95,-1.4) and (3.31,-0.3) .. (0,0) .. controls (3.31,0.3) and (6.95,1.4) .. (10.93,3.29)   ;
\draw  [dash pattern={on 4.5pt off 4.5pt}]  (105,157) -- (202.5,233.27) ;
\draw [shift={(204.08,234.51)}, rotate = 218.04] [color={rgb, 255:red, 0; green, 0; blue, 0 }  ][line width=0.75]    (10.93,-3.29) .. controls (6.95,-1.4) and (3.31,-0.3) .. (0,0) .. controls (3.31,0.3) and (6.95,1.4) .. (10.93,3.29)   ;

\end{tikzpicture}

    \caption{The closed curve $\phi(\gamma)$ does not have intersection of this type}
    \label{fig: intro.hexagon}
\end{figure}
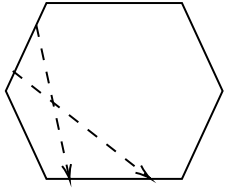

\paragraph{Question.} Can the ergodicity assumption in Theorem \ref{thm: main.entropy} be removed? Specifically, if geodesic currents $\{ C_n\}$ satisfy $i(C_n,C_n)\to 0$ as $n \to \infty$, does it follow that $h_{X}(C_n) \to 0$?

\paragraph{Notes and references.}

A qualitative analog of Theorem~\ref{thm: main.entropy}, formulated in terms of the Hausdorff dimension of geodesic support rather than geodesic current entropy, can be deduced from the work of Sapir~\cite{Sapir.BS}. More precisely, she showed that the Hausdorff dimension of the set of geodesic rays with small self-intersection relative to their length is small. Additional bounds on $P_{\epsilon}(T)$ are also obtained in \cite{sapir-bounds}. However, the estimates available there are insufficient to imply Theorem \ref{thm: main.entropy}. In particular, those results give a constant upper bound on the limit appearing in Theorem \ref{thm: entropy<P} independent of $\epsilon$.

\paragraph{Acknowledgements.} I would like to thank Curt T. McMullen for his continuous help, invaluable discussions, and insightful suggestions related to the results of this paper. I also thank Yongquan Zhang for helpful comments on the exposition of this paper.

\section{Geodesic currents}\label{sec: geod}

In this section, we briefly explain some basics of the theory of geodesic currents. For a reference, see \cite{Bon.gc.3}\cite{Bon.gc.Tech}.

Let $X:=\mathbb{H}/\Gamma$ be a compact hyperbolic surface. Let $T_1(X)$ denote the unit tangent bundle of $X$. A geodesic current $C$ is a Borel measure on $T_1(X)$ that is invariant under both the geodesic flow and the involution map (which is the map that sends $v$ to $-v$ for $v\in T_1(X)$).  
 
 Let $\Cg$ be the space of all geodesic currents.
 
 Here are some important examples of geodesic currents:
 \begin{itemize}
     \item \textbf{Closed geodesics.}
Let $\gamma$ be a closed geodesic on $X$. The unit tangent vectors along $\gamma$ determine a subset $\widetilde{\gamma} \subset T_1(X)$. 
The arclength measure on $\gamma$ induces a measure on $T_1(X)$ supported on $\widetilde{\gamma}$, which is invariant under both the geodesic flow and the involution map $v \mapsto -v$ for $v \in T_1(X)$.

Weighted closed geodesics form a dense subset of $\mathcal{C}_g$. 
Bonahon showed that the notions of hyperbolic length and intersection number for closed geodesics extend continuously to all geodesic currents. 
For $C_1, C_2 \in \mathcal{C}_g$, we denote their intersection number by $i(C_1,C_2)$ and the hyperbolic length of $C_1$ by $\ell_X(C_1)$, where
$$
\ell_X(C_1) = C_1(T_1(X)),
$$
which is the total volume of $T_1(X)$ with respect to $C_1$.

An explicit definition of the intersection form on $\mathcal{C}_g$ can also be found in~\cite{Bon.gc.3}.

    \item \textbf{Measured laminations.} These are exactly geodesic currents with zero self-intersection number.
    
     \item \textbf{Liouville measure.} This is denoted by $L_X$ and is the canonical volume measure on $T_1(X)$ obtained from the hyperbolic metric on $X$.  
 \end{itemize}

 Another example of a geodesic current arises from random simple closed curves; see \cite{Tina-McMullen} for details.

\paragraph{Intersection number.}  Given closed geodesics $\gamma_1$ and $\gamma_2$, their intersection number $i(\gamma_1,\gamma_2)$ is defined as the total number of transverse intersection points, counted with multiplicity. Here, multiplicity means that an intersection point $p$ is assigned a weight $m_p>1$ if $\gamma_1$ or $\gamma_2$ passes through $p$ more than once. By a \emph{simple intersection} we mean an intersection point $p$ at which each curve passes exactly once and the curves meet transversely at $p$. The multiplicity $m_p$ is defined as follows: by homotopically perturbing the curves in a sufficiently small neighborhood $U$ of $p$, one can arrange all intersections to be simple in $U$. The minimal number of simple intersection points within $U$ obtained in this way is defined as $m_p$.

The intersection number can be bounded above in terms of hyperbolic length. More precisely, for any $C_1, C_2 \in \mathcal{C}_g$, we have
\[
i(C_1,C_2) \le I(X)\,\ell_X(C_1)\,\ell_X(C_2),
\]
for a constant $I(X)$.
See~\cite{int-sys-T,Basmj} for further details. In particular, in Theorem~\ref{thm: main.entropy}, when $\ell_X(C)=1$, this estimate shows that $i(C,C)$ ranges between $0$ and $I(X)$.

  The variable $I(X)$ is called the \textit{interaction strength} of $X$ and is defined in \cite{int-sys-T} as
  $$
  I(X):=\sup \limits_{C_1,C_2 \in \Cg} \frac{i(C_1,C_2)}{\ell_X(C_1)\ell_X(C_2)}.
  $$
 We have $I(X)\leq 4/\sys(X)^2$ \cite[Prop.~2.4]{int-sys-T}. Moreover, for fixed genus $g>1$, we can see that $I(X) \to \infty$, as $\sys(X)\to 0$. 

\section{Counting closed geodesics}\label{sec: count}

In this section, we prove Theorem \ref{thm: P-closed-geod}.

\paragraph{Hexagon decomposition.}
Consider a pair of pants $S$ (a surface of genus $0$ with three boundaries), the boundaries of $S$ are called \textit{cuffs} of $S$. Assume that $S$ is equipped with a hyperbolic metric. The shortest arcs between different cuffs are called \textit{seams} of $S$, and they split $S$ into two isomorphic hexagons.

 Let $X$ be a compact surface of genus $g$. Let $P=\{p_1,\dots,p_{3g-3}\}$ be a maximal set of simple closed disjoint curves on $X$. They divided $X$ into $2g-2$ pairs of pants. A \emph{multi-curve} is a finite sum of pairwise disjoint or identical simple closed curves. Now consider a multi-curve $\beta$ that intersects each $p_i$ at two points and divides each pair of pants into two hexagons. The union of all $p_i$ and $\beta$ gives a hexagon decomposition of $X$ and split $X$ into hexagons $H_1,\dots,H_{4g-4}$.

 When $X$ is equipped with a hyperbolic metric, we know that there is a pants decomposition such that $\ell_X(p_i) \leq L_g$ for $i=1, \dots, 3g-3$, where $L_g$ is the \textit{Bers' constant}. Moreover, we have $L_g \leq 26(g-1)$ \cite[Thm. 5.1.2]{Bsr}. We call such a decomposition a \textit{proper pants decomposition} of $X$.

A \textit{proper hexagon decomposition} is constructed as follows. On each curve $p_i$, we mark two points that divide $p_i$ into two equal arcs. It is well known that, for a pair of pants, the endpoints of its seams on each cuff also split the cuff into two equal arcs.
Therefore, by applying suitable twists around the cuffs, each of magnitude less than a full twist, we may arrange that the endpoints of the seams become coincide with the marked points on the curves $p_i$. 

We then take the geodesic representatives of the resulting arcs connecting the marked points, by keeping the endpoints fixed on the cuffs. With this choice of twists, the resulting hexagons meet along full edges rather than partial segments. In particular, no vertex of one hexagon lies in the interior of an edge of another, and vertices can coincide only with other vertices.

This construction yields a piecewise geodesic curve $\beta$ and determines a \emph{proper hexagon decomposition} of $X$. 
Note that the curve $\beta$ needs not be connected.

\paragraph{Symbolic coding of closed curves.}

 Label the edges of the hexagons $H_i$ by $W=\{e_1,\dots,e_n\}$ for $n=24g-24$. Note that different labels may refer to the same arc on $X$ but from different hexagons.

Consider an oriented closed geodesic $\gamma$ with a marked point. Then we can assign a word $\omega(\gamma)$ to $\gamma$ with letters in $W$. The word presents the edges of the hexagons that $\gamma$ outgoing from them in order. See Figure \ref{fig: word}. 

If $\gamma$ passes through a vertex, we homotopically perturb $\gamma$ slightly so that it avoids the vertex, and then define $\omega(\gamma)$ using this perturbed curve. 
Thus, this notation involves a mild abuse: the word $\omega(\gamma)$ may not be uniquely determined. 
In such cases, we choose one representative arbitrarily.

Nevertheless, this ambiguity does not affect our arguments, since if $\omega(\gamma_1)=\omega(\gamma_2)$, then the curves $\gamma_1$ and $\gamma_2$ are homotopic which is good enough for our argument.

Define $\ell_{com}(\gamma)$, the combinatorial length of $\gamma$, as the number of letters of $\omega(\gamma)$.

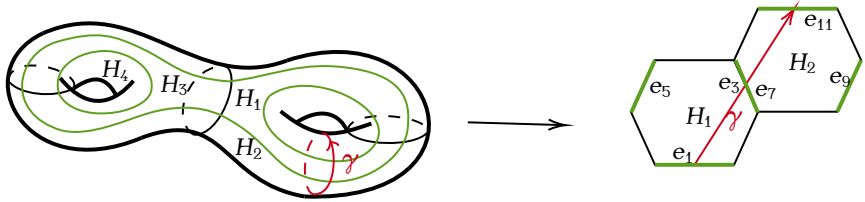
\begin{figure}[h]
    \centering

\tikzset{every picture/.style={line width=0.75pt}} 

\begin{tikzpicture}[x=0.75pt,y=0.75pt,yscale=-0.8,xscale=0.8]

\draw [line width=1.5]    (75.2,81.8) .. controls (135.2,55.8) and (171.2,120.8) .. (229,99) ;
\draw [line width=1.5]    (229,99) .. controls (327.2,66.8) and (363.2,188.8) .. (240.2,184.8) ;
\draw [line width=1.5]    (137.2,147.8) .. controls (191.2,127.8) and (187.2,175.8) .. (240.2,184.8) ;
\draw [line width=1.5]    (75.2,81.8) .. controls (26.2,104.8) and (63.2,178.8) .. (137.2,147.8) ;
\draw [line width=1.5]    (86,110) .. controls (105.2,133.8) and (122.2,123.8) .. (132.2,111.8) ;
\draw [line width=1.5]    (225,131) .. controls (249.2,150.8) and (261.2,145.8) .. (282.2,138.8) ;
\draw [line width=1.5]    (94.2,116.8) .. controls (103.2,110.8) and (113.2,105.8) .. (121.2,120.8) ;
\draw [line width=1.5]    (235,138) .. controls (251.2,127.8) and (262.2,132.8) .. (267.2,143.8) ;
\draw [color={rgb, 255:red, 0; green, 0; blue, 0 }  ,draw opacity=1 ]   (191,103) .. controls (200.2,109.8) and (187.2,146.8) .. (170.2,145.8) ;
\draw [color={rgb, 255:red, 0; green, 0; blue, 0 }  ,draw opacity=1 ] [dash pattern={on 4.5pt off 4.5pt}]  (170.2,145.8) .. controls (154.2,135.8) and (168.2,99.8) .. (191,103) ;
\draw [color={rgb, 255:red, 0; green, 0; blue, 0 }  ,draw opacity=1 ]   (53.2,111.8) .. controls (60.2,120.8) and (82.2,122.8) .. (94.2,116.8) ;
\draw [color={rgb, 255:red, 0; green, 0; blue, 0 }  ,draw opacity=1 ] [dash pattern={on 4.5pt off 4.5pt}]  (53.2,111.8) .. controls (55.2,92.8) and (101.2,105.8) .. (94.2,116.8) ;
\draw [color={rgb, 255:red, 0; green, 0; blue, 0 }  ,draw opacity=1 ]   (267.2,143.8) .. controls (274.2,152.8) and (308.2,153.8) .. (317.2,144.8) ;
\draw [color={rgb, 255:red, 0; green, 0; blue, 0 }  ,draw opacity=1 ] [dash pattern={on 4.5pt off 4.5pt}]  (267.2,143.8) .. controls (267.2,127.8) and (322.2,133.8) .. (317.2,144.8) ;
\draw [color={rgb, 255:red, 95; green, 159; blue, 31 }  ,draw opacity=1 ]   (85.27,120) .. controls (95.27,140) and (139.27,133) .. (143.2,118.8) .. controls (147.13,104.6) and (129.27,93) .. (114.27,93) .. controls (99.27,93) and (80.27,102) .. (85.27,120) -- cycle ;
\draw [color={rgb, 255:red, 95; green, 159; blue, 31 }  ,draw opacity=1 ]   (87.1,143.16) .. controls (105.2,147.33) and (128,136.11) .. (142.05,132.13) .. controls (156.1,128.16) and (172.92,126.58) .. (188.01,133.87) ;
\draw [color={rgb, 255:red, 95; green, 159; blue, 31 }  ,draw opacity=1 ]   (192.1,112.16) .. controls (162.01,109.87) and (156.08,100.27) .. (138.09,91.71) .. controls (120.1,83.16) and (110.09,79.71) .. (84.1,87.16) .. controls (58.11,94.61) and (50.2,132.33) .. (87.1,143.16) ;
\draw [color={rgb, 255:red, 95; green, 159; blue, 31 }  ,draw opacity=1 ]   (188.01,133.87) .. controls (221.01,150.87) and (218.89,168.4) .. (247.89,174.4) .. controls (276.89,180.4) and (302.62,160.04) .. (304.89,143.4) .. controls (307.17,126.76) and (304.72,109.59) .. (270.66,104.05) .. controls (236.6,98.51) and (219.85,109.48) .. (192.1,112.16) ;
\draw [color={rgb, 255:red, 95; green, 159; blue, 31 }  ,draw opacity=1 ]   (217.89,121.4) .. controls (197.89,147.4) and (266.89,180.4) .. (285,151) ;
\draw [color={rgb, 255:red, 95; green, 159; blue, 31 }  ,draw opacity=1 ]   (217.89,121.4) .. controls (243.89,102.4) and (297.89,129.4) .. (285,151) ;
\draw [color={rgb, 255:red, 208; green, 2; blue, 27 }  ,draw opacity=1 ]   (246.46,183.34) .. controls (262.46,183.34) and (260.46,145.34) .. (252,145) ;
\draw [color={rgb, 255:red, 208; green, 2; blue, 27 }  ,draw opacity=1 ] [dash pattern={on 4.5pt off 4.5pt}]  (246.46,183.34) .. controls (237.46,184.34) and (237.46,145.34) .. (252,145) ;
\draw    (343,138) -- (402.82,139.74) ;
\draw [shift={(404.82,139.8)}, rotate = 181.67] [color={rgb, 255:red, 0; green, 0; blue, 0 }  ][line width=0.75]    (10.93,-3.29) .. controls (6.95,-1.4) and (3.31,-0.3) .. (0,0) .. controls (3.31,0.3) and (6.95,1.4) .. (10.93,3.29)   ;
\draw   (446,131.8) -- (460.97,98.8) -- (510.85,98.8) -- (525.82,131.8) -- (510.85,164.8) -- (460.97,164.8) -- cycle ;
\draw   (510.85,98.8) -- (525.87,66.3) -- (575.93,66.3) -- (590.95,98.8) -- (575.93,131.3) -- (525.87,131.3) -- cycle ;
\draw [color={rgb, 255:red, 208; green, 2; blue, 27 }  ,draw opacity=1 ]   (485.91,164.8) -- (548.15,67.31) ;
\draw [shift={(549.23,65.62)}, rotate = 122.56] [color={rgb, 255:red, 208; green, 2; blue, 27 }  ,draw opacity=1 ][line width=0.75]    (10.93,-3.29) .. controls (6.95,-1.4) and (3.31,-0.3) .. (0,0) .. controls (3.31,0.3) and (6.95,1.4) .. (10.93,3.29)   ;
\draw  [color={rgb, 255:red, 208; green, 2; blue, 27 }  ,draw opacity=1 ][line width=0.75] [line join = round][line cap = round] (254.25,155.43) .. controls (255.49,155.43) and (256.25,151.7) .. (256.25,150.43) ;
\draw  [color={rgb, 255:red, 208; green, 2; blue, 27 }  ,draw opacity=1 ][line width=0.75] [line join = round][line cap = round] (256.25,150.43) .. controls (257.52,150.43) and (261.25,151.2) .. (261.25,152.43) ;
\draw [color={rgb, 255:red, 95; green, 159; blue, 31 }  ,draw opacity=1 ][line width=1.5]    (460.97,164.8) -- (510.85,164.8) ;
\draw [color={rgb, 255:red, 95; green, 159; blue, 31 }  ,draw opacity=1 ][line width=1.5]    (525.82,131.8) -- (511.8,98.8) ;
\draw [color={rgb, 255:red, 95; green, 159; blue, 31 }  ,draw opacity=1 ][line width=1.5]    (446,131.8) -- (460.97,98.8) ;
\draw [color={rgb, 255:red, 95; green, 159; blue, 31 }  ,draw opacity=1 ][line width=1.5]    (575.93,131.3) -- (590.95,98.8) ;
\draw [color={rgb, 255:red, 95; green, 159; blue, 31 }  ,draw opacity=1 ][line width=1.5]    (526.05,66.3) -- (575.93,66.3) ;

\draw (194.1,115.56) node [anchor=north west][inner sep=0.75pt]  [font=\footnotesize,color={rgb, 255:red, 0; green, 0; blue, 0 }  ,opacity=1 ]  {$H_{1}$};
\draw (194,146.4) node [anchor=north west][inner sep=0.75pt]  [font=\footnotesize,color={rgb, 255:red, 0; green, 0; blue, 0 }  ,opacity=1 ]  {$H_{2}$};
\draw (147,106.4) node [anchor=north west][inner sep=0.75pt]  [font=\footnotesize,color={rgb, 255:red, 0; green, 0; blue, 0 }  ,opacity=1 ]  {$H_{3}$};
\draw (110,97) node [anchor=north west][inner sep=0.75pt]  [font=\footnotesize,color={rgb, 255:red, 0; green, 0; blue, 0 }  ,opacity=1 ]  {$H_{4}$};
\draw (261,156.4) node [anchor=north west][inner sep=0.75pt]  [color={rgb, 255:red, 208; green, 2; blue, 27 }  ,opacity=1 ]  {$\gamma $};
\draw (478.1,126.56) node [anchor=north west][inner sep=0.75pt]  [font=\footnotesize,color={rgb, 255:red, 0; green, 0; blue, 0 }  ,opacity=1 ]  {$H_{1}$};
\draw (543.1,92.56) node [anchor=north west][inner sep=0.75pt]  [font=\footnotesize,color={rgb, 255:red, 0; green, 0; blue, 0 }  ,opacity=1 ]  {$H_{2}$};
\draw (470,152.4) node [anchor=north west][inner sep=0.75pt]  [font=\footnotesize]  {$e_{1}$};
\draw (499,108.4) node [anchor=north west][inner sep=0.75pt]  [font=\footnotesize]  {$e_{3}$};
\draw (455.48,110.7) node [anchor=north west][inner sep=0.75pt]  [font=\footnotesize]  {$e_{5}$};
\draw (522,112.4) node [anchor=north west][inner sep=0.75pt]  [font=\footnotesize]  {$e_{7}$};
\draw (569,104.4) node [anchor=north west][inner sep=0.75pt]  [font=\footnotesize]  {$e_{9}$};
\draw (553,68.4) node [anchor=north west][inner sep=0.75pt]  [font=\footnotesize]  {$e_{11}$};
\draw (122,144.4) node [anchor=north west][inner sep=0.75pt]    {$$};
\draw (503,129.4) node [anchor=north west][inner sep=0.75pt]  [color={rgb, 255:red, 208; green, 2; blue, 27 }  ,opacity=1 ]  {$\gamma $};

\end{tikzpicture}

    \caption{In this hexagon decompisition of a genus $2$ surface, $\omega(\gamma)=e_1e_7$ and $\ell_{com}(\gamma)=2$.}
    \label{fig: word}
\end{figure}

\begin{Lem}\label{lemma: com.length}
For any oriented closed geodesic $\gamma$ with a marked point, we have 
$$
\ell_{com}(\gamma)\leq c(X) \ell_X(\gamma),
$$
where $c(X)\leq I(X)(21g-21)L_g+12/\sys(X)$. Here, $I(X)$ and $L_g$ are the interaction strength and the Bers' constant, respectively.
\end{Lem}

\begin{proof}

As explained above, we may assume that $\gamma$ does not pass through any vertex of the hexagons.
The intersection points of $\gamma$ with the edges of the hexagons divide $\gamma$ into $\ell_{\mathrm{com}}(\gamma)$ geodesic subarcs.
Each such subarc either has at least one endpoint on a cuff, or both of its endpoints lie on $\beta$.

Our goal is to obtain an upper bound on the number of these intersection points. Consider a block of subarcs corresponding to each $7$ consecutive intersection points along $\gamma$ (including $6$ consecutive subarcs).
Either at least one of these points lies on a cuff, or all seven points lie on $\beta$.
The number of the blocks of the first type is bounded above by $7i(\gamma,p)$.

Now consider a subarc $\alpha$ of the second type.
Since $\alpha$ does not intersect any cuff, it must be entirely contained in a single pair of pants $D$.
Let $K_1$ and $K_2$ be the two hexagons in $D$ determined by the seams of $D$.
Observe that any edge $e_i$ contained in $D$ has at most one connected component in each of $K_1$ and $K_2$. There are three edges $e_i$ in the interior of $D$. Among the seven intersection points of $\alpha$ with these edges, there must exist two points, say $q_1$ and $q_2$, lying on the same edge $e_i$ within the same hexagon $K_j$. Without loss of generality, assume $j=1$; see Figure~\ref{fig: seam}.

Since two geodesic arcs inside a hexagon can intersect at most once (otherwise a bigon would be formed), the subarc of $\alpha$ between $q_1$ and $q_2$ must start in $K_1$, pass through $K_2$, and return to $K_1$.
In particular, the portion of $\alpha$ inside $K_2$ connects two distinct seams of $D$, and hence has length at least $\sys(X)/2$. See Figure \ref{fig: seam}.
It follows that the number of subarcs of this second type is bounded above by $6\ell_X(\gamma)/(\sys(X)/2)$.

\begin{figure}[h]
    \centering

\tikzset{every picture/.style={line width=0.75pt}} 

\begin{tikzpicture}[x=0.75pt,y=0.75pt,yscale=-0.9,xscale=0.9]

\draw   (300,92.9) .. controls (300,80.69) and (316.83,70.8) .. (337.6,70.8) .. controls (358.37,70.8) and (375.2,80.69) .. (375.2,92.9) .. controls (375.2,105.11) and (358.37,115) .. (337.6,115) .. controls (316.83,115) and (300,105.11) .. (300,92.9) -- cycle ;
\draw   (163,93) .. controls (163,81.95) and (178.67,73) .. (198,73) .. controls (217.33,73) and (233,81.95) .. (233,93) .. controls (233,104.05) and (217.33,113) .. (198,113) .. controls (178.67,113) and (163,104.05) .. (163,93) -- cycle ;
\draw   (234,193) .. controls (234,181.95) and (249.67,173) .. (269,173) .. controls (288.33,173) and (304,181.95) .. (304,193) .. controls (304,204.05) and (288.33,213) .. (269,213) .. controls (249.67,213) and (234,204.05) .. (234,193) -- cycle ;
\draw    (168.2,103.8) .. controls (206.2,128.8) and (227.2,160.8) .. (234,193) ;
\draw    (233,93) .. controls (243.2,148.8) and (296.2,124.8) .. (300,92.9) ;
\draw    (304,193) .. controls (298.2,159.8) and (330.2,137.8) .. (368,106) ;
\draw [color={rgb, 255:red, 208; green, 2; blue, 27 }  ,draw opacity=1 ]   (218.2,151.8) .. controls (231.2,166.8) and (255.2,158.8) .. (269,173) ;
\draw [color={rgb, 255:red, 208; green, 2; blue, 27 }  ,draw opacity=1 ] [dash pattern={on 4.5pt off 4.5pt}]  (218.2,151.8) .. controls (239.2,123.8) and (206.2,109.8) .. (198,73) ;
\draw [color={rgb, 255:red, 95; green, 159; blue, 31 }  ,draw opacity=1 ]   (255.2,165.8) .. controls (269.2,160.8) and (275.2,131.8) .. (264.2,127.8) ;
\draw [color={rgb, 255:red, 95; green, 159; blue, 31 }  ,draw opacity=1 ] [dash pattern={on 4.5pt off 4.5pt}]  (229.2,171.8) .. controls (220.2,166.8) and (241.2,122.8) .. (264.2,127.8) ;
\draw [color={rgb, 255:red, 95; green, 159; blue, 31 }  ,draw opacity=1 ]   (229.2,171.8) .. controls (234.2,177.8) and (243.2,171.8) .. (246.2,162.8) ;
\draw  [line width=3] [line join = round][line cap = round] (256.2,163.8) .. controls (256.2,164.13) and (256.2,164.47) .. (256.2,164.8) ;
\draw  [line width=3] [line join = round][line cap = round] (244.2,163.8) .. controls (244.67,163.8) and (245.2,164.33) .. (245.2,164.8) ;

\draw (262,156.4) node [anchor=north west][inner sep=0.75pt]  [font=\scriptsize]  {$q_{1}$};
\draw (272,131.4) node [anchor=north west][inner sep=0.75pt]  [font=\small]  {$\alpha $};
\draw (325,160) node [anchor=north west][inner sep=0.75pt]  [font=\small]  {$D$};
\draw (237,146.4) node [anchor=north west][inner sep=0.75pt]  [font=\scriptsize]  {$q_{2}$};
\draw (208,116.4) node [anchor=north west][inner sep=0.75pt]  [font=\small]  {$e_i$};

\end{tikzpicture}

    \caption{The subarc $\alpha$ has two intersection points $q_1,q_2$ on an edge $e_i$ inside a hexagon obtained from the seams of the pair of pants, this arc has length $\geq \sys(X)/2$.}
    \label{fig: seam}
\end{figure}
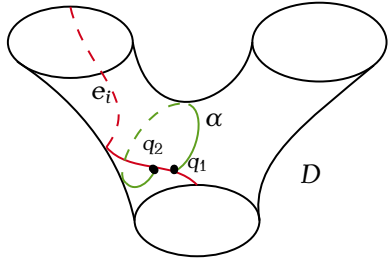

Combining the two cases, we obtain
\[
\ell_{\mathrm{com}}(\gamma)
\leq 7\, i(\gamma,p) + \frac{12\,\ell_X(\gamma)}{\sys(X)}.
\]
Using the intersection bound $i(\gamma,p) \leq I(X)\,\ell_X(\gamma)\,\ell_X(p)$ (see \S \ref{sec: geod}), we conclude that
\[
\ell_{\mathrm{com}}(\gamma)
\leq 7\, I(X)\,\ell_X(\gamma)\,\ell_X(p) + \frac{12\,\ell_X(\gamma)}{\sys(X)}
\leq I(X)\,\ell_X(\gamma)\,(21g-21)\,L_g + \frac{12\,\ell_X(\gamma)}{\sys(X)},
\]
where $I(X)$ denotes the intersection strength (see \S\ref{sec: geod}) and $L_g$ is the Bers' constant.

\end{proof}

See Remark \ref{remark: bc.bound} for a bound on $c(X)$.
\paragraph{Intersection points.}

Let $\alpha$ and $\beta$ be simple oriented geodesic arcs in hexagon $H_i$ whose endpoints are on the edges of $H_i$ (but not on the vertices). Assume that $\alpha,\beta$ intersect inside $H_i$. The configuration of $\alpha,\beta$ in $H_i$, based on their endpoints, is of one of the following types (see Figure \ref{fig: int.type}):
\begin{itemize}
\item \emph{Type 1.} Neither their starting points nor their terminal points are on the same edge.
\item \emph{Type 2 (a).} Their starting points are on the same edge but their terminal points are not. 
\item \emph{Type 2 (b).} Their terminal points are on the same edge but their starting points are not.
\item \emph{Type 3.} Their initial points lie on the same edge, and their terminal points also lie on the same edge.

\end{itemize}

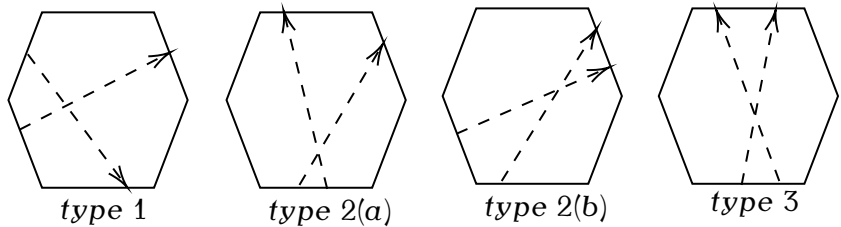
\begin{figure}[h]
    \centering

\tikzset{every picture/.style={line width=0.75pt}} 

\begin{tikzpicture}[x=0.75pt,y=0.75pt,yscale=-1,xscale=1]

\draw   (100,167.27) -- (116.95,123) -- (173.45,123) -- (190.4,167.27) -- (173.45,211.53) -- (116.95,211.53) -- cycle ;
\draw   (210,167.27) -- (226.95,123) -- (283.45,123) -- (300.4,167.27) -- (283.45,211.53) -- (226.95,211.53) -- cycle ;
\draw   (319,165.27) -- (335.95,121) -- (392.45,121) -- (409.4,165.27) -- (392.45,209.53) -- (335.95,209.53) -- cycle ;
\draw   (428,165.27) -- (444.95,121) -- (501.45,121) -- (518.4,165.27) -- (501.45,209.53) -- (444.95,209.53) -- cycle ;
\draw  [dash pattern={on 4.5pt off 4.5pt}]  (109.78,144.13) -- (158.58,209.53) ;
\draw [shift={(159.78,211.13)}, rotate = 233.27] [color={rgb, 255:red, 0; green, 0; blue, 0 }  ][line width=0.75]    (10.93,-3.29) .. controls (6.95,-1.4) and (3.31,-0.3) .. (0,0) .. controls (3.31,0.3) and (6.95,1.4) .. (10.93,3.29)   ;
\draw  [dash pattern={on 4.5pt off 4.5pt}]  (105.78,182.13) -- (180,144.05) ;
\draw [shift={(181.78,143.13)}, rotate = 152.84] [color={rgb, 255:red, 0; green, 0; blue, 0 }  ][line width=0.75]    (10.93,-3.29) .. controls (6.95,-1.4) and (3.31,-0.3) .. (0,0) .. controls (3.31,0.3) and (6.95,1.4) .. (10.93,3.29)   ;
\draw  [dash pattern={on 4.5pt off 4.5pt}]  (246.78,210.13) -- (287.95,140.5) ;
\draw [shift={(288.97,138.77)}, rotate = 120.59] [color={rgb, 255:red, 0; green, 0; blue, 0 }  ][line width=0.75]    (10.93,-3.29) .. controls (6.95,-1.4) and (3.31,-0.3) .. (0,0) .. controls (3.31,0.3) and (6.95,1.4) .. (10.93,3.29)   ;
\draw  [dash pattern={on 4.5pt off 4.5pt}]  (260.78,212.13) -- (239.25,124.08) ;
\draw [shift={(238.78,122.13)}, rotate = 76.26] [color={rgb, 255:red, 0; green, 0; blue, 0 }  ][line width=0.75]    (10.93,-3.29) .. controls (6.95,-1.4) and (3.31,-0.3) .. (0,0) .. controls (3.31,0.3) and (6.95,1.4) .. (10.93,3.29)   ;
\draw  [dash pattern={on 4.5pt off 4.5pt}]  (326,184.27) -- (401.13,151.57) ;
\draw [shift={(402.97,150.77)}, rotate = 156.48] [color={rgb, 255:red, 0; green, 0; blue, 0 }  ][line width=0.75]    (10.93,-3.29) .. controls (6.95,-1.4) and (3.31,-0.3) .. (0,0) .. controls (3.31,0.3) and (6.95,1.4) .. (10.93,3.29)   ;
\draw  [dash pattern={on 4.5pt off 4.5pt}]  (348.78,208.13) -- (395.91,132.47) ;
\draw [shift={(396.97,130.77)}, rotate = 121.92] [color={rgb, 255:red, 0; green, 0; blue, 0 }  ][line width=0.75]    (10.93,-3.29) .. controls (6.95,-1.4) and (3.31,-0.3) .. (0,0) .. controls (3.31,0.3) and (6.95,1.4) .. (10.93,3.29)   ;
\draw  [dash pattern={on 4.5pt off 4.5pt}]  (469.78,210.13) -- (486.58,123.74) ;
\draw [shift={(486.97,121.77)}, rotate = 101.01] [color={rgb, 255:red, 0; green, 0; blue, 0 }  ][line width=0.75]    (10.93,-3.29) .. controls (6.95,-1.4) and (3.31,-0.3) .. (0,0) .. controls (3.31,0.3) and (6.95,1.4) .. (10.93,3.29)   ;
\draw  [dash pattern={on 4.5pt off 4.5pt}]  (488.78,208.13) -- (457.66,123.65) ;
\draw [shift={(456.97,121.77)}, rotate = 69.78] [color={rgb, 255:red, 0; green, 0; blue, 0 }  ][line width=0.75]    (10.93,-3.29) .. controls (6.95,-1.4) and (3.31,-0.3) .. (0,0) .. controls (3.31,0.3) and (6.95,1.4) .. (10.93,3.29)   ;

\draw (124,215.4) node [anchor=north west][inner sep=0.75pt]    {$type\ 1$};
\draw (228.95,214.93) node [anchor=north west][inner sep=0.75pt]    {$type\ 2( a)$};
\draw (337.95,212.93) node [anchor=north west][inner sep=0.75pt]    {$type\ 2( b)$};
\draw (452,211.4) node [anchor=north west][inner sep=0.75pt]    {$type\ 3$};

\end{tikzpicture}

    \caption{Different types of intersection between two arcs}
    \label{fig: int.type}
\end{figure}

Let $\gamma_1$ and $\gamma_2$ be oriented closed curves on $X$ that do not pass through the vertices of the hexagons.
We say that $\gamma_1$ and $\gamma_2$ are \emph{linearly homotopic} if there exists a homotopy $h$ between them such that points lying on an edge are moved only along that same edge throughout the homotopy.
Equivalently, if $\alpha$ is a subarc of $\gamma_1$ with endpoints on edges $b$ and $b'$, then $h$ moves the endpoints of $\alpha$ along $b$ and $b'$, and deforms $\alpha$ into a subarc connecting the same pair of edges.

We say a closed curve is \textit{proper} if it does not pass through the vertices of the hexagos and its subarcs inside hexagons are geodesics.
The key step of the proof is the following theorem:

\begin{Thm}\label{thm: linear.homotop}
Let $\gamma$ be a proper oriented closed curve.
Then there exists a proper oriented closed curve $\phi(\gamma)$, linearly homotopic to $\gamma$, with no intersections of type $2(a)$ or type $3$.
Moreover, the number of intersection points of type $2$ for $\phi(\gamma)$ is not greater than that for $\gamma$.

\end{Thm}

\begin{proof}

We construct $\phi(\gamma)$ to be linearly homotopic to $\gamma$ as follows. 
Fix a hexagon $H_j$ and an edge $e_i \subset H_j$, and consider the arcs of $\gamma$ exiting $H_j$ through $e_i$. 
Assume that these outgoing arcs intersect $e_i$ at $n$ distinct points. 
Among all permutations of these $n$ intersection points, we will specify a distinguished one, called the \emph{proper ordering}. 
Our goal is to choose this ordering so that, if the intersection points on every edge are arranged according to the proper ordering, then no intersections of types $2(a)$ and $3$ occur. Moreover, we can assume that the subarcs of $\phi(\gamma)$ in each hexagon are geodesics.

The orientation of $X$ induces an orientation on each hexagon and on each of its edges. For each hexagon $H_j$, this, in turn, determines a cyclic ordering of its edges. After removing an edge $e_i \subset H_j$, this cyclic ordering induces a linear ordering on the remaining edges.

Note that when two hexagons share an edge, the orientations they induce on that edge are opposite. Therefore, it is necessary to specify a choice of hexagon whenever we refer to the orientation of an edge. 
Throughout this discussion, we adopt the convention that, for any intersection point with an edge, the relevant orientation is the one induced on that edge by the hexagon containing the forward arc issuing from the intersection point.

\textit{The proper ordering} is defined as follows. Let $\mathbb{D}$ denote the hyperbolic disk, identified with the universal cover of $X$. 
Let $\widetilde{H}_j$ be a lift of $H_j$ to $\mathbb{D}$, and $\widetilde{e}_i$ be the corresponding lift of $e_i$ contained in $\widetilde{H}_j$. 
For an intersection point $p \in \widetilde{e}_i \cap \widetilde{\gamma}$, where $\widetilde{\gamma}$ is a lift of $\gamma$, define $f^{+}(p)$ to be the endpoint of $\widetilde{\gamma}$ on $\partial \mathbb{D}$ in the forward direction.

Note that if $\widetilde{\gamma}_1$ and $\widetilde{\gamma}_2$ are lifts of two linearly homotopic curves, then for the corresponding intersection points
$p_1 = \widetilde{e}_i \cap \widetilde{\gamma}_1$ and 
$p_2 = \widetilde{e}_i \cap \widetilde{\gamma}_2$, we have
\[
f^{+}(p_1) = f^{+}(p_2).
\]

The geodesic $\widetilde{e}_i$ divides $\partial \mathbb{D}$ into two intervals. 
For the intersection points $p_1,\dots,p_n$ arising from outgoing arcs of $\gamma$ on $e_i$, the points
$f^{+}(p_1),\dots,f^{+}(p_n)$ all lie on the same boundary interval, namely the one corresponding to the side on which $\widetilde{H}_j$ lies (see Figure~\ref{fig: universal-hexagon}). We define the \emph{proper ordering} of the points $p_1,\dots,p_n$ on $e_i$ to be the one that is compatible with the order of
$f^{+}(p_1),\dots,f^{+}(p_n)$ along this boundary interval of $\partial \mathbb{D}$. 
Equivalently, the permutation of the intersection points on $e_i$ that agrees with the order of their forward endpoints on $\partial \mathbb{D}$ is defined as \textit{the proper ordering}; see Figure~\ref{fig: universal-hexagon}.

\begin{figure}[h]
    \centering

\tikzset{every picture/.style={line width=0.75pt}} 

\begin{tikzpicture}[x=0.75pt,y=0.75pt,yscale=-0.9,xscale=0.9]

\draw   (95,178.9) .. controls (95,89.93) and (167.13,17.8) .. (256.1,17.8) .. controls (345.07,17.8) and (417.2,89.93) .. (417.2,178.9) .. controls (417.2,267.87) and (345.07,340) .. (256.1,340) .. controls (167.13,340) and (95,267.87) .. (95,178.9) -- cycle ;
\draw  [draw opacity=0] (167.69,312.58) .. controls (167.55,310.8) and (167.49,309.01) .. (167.5,307.2) .. controls (167.81,265.79) and (208.66,232.53) .. (258.74,232.9) .. controls (308.82,233.28) and (349.17,267.15) .. (348.86,308.56) .. controls (348.86,308.91) and (348.85,309.26) .. (348.84,309.61) -- (258.18,307.88) -- cycle ; \draw   (167.69,312.58) .. controls (167.55,310.8) and (167.49,309.01) .. (167.5,307.2) .. controls (167.81,265.79) and (208.66,232.53) .. (258.74,232.9) .. controls (308.82,233.28) and (349.17,267.15) .. (348.86,308.56) .. controls (348.86,308.91) and (348.85,309.26) .. (348.84,309.61) ;  
\draw    (269,233) .. controls (268.18,217.91) and (275.18,205.91) .. (288,197) ;
\draw    (177.2,199.6) .. controls (203.2,202.6) and (217.2,215.6) .. (227.2,236.6) ;
\draw    (208,210) .. controls (219.18,200.91) and (227.2,188.6) .. (229.2,176.6) ;
\draw    (265.18,171.91) .. controls (273.18,184.91) and (279.18,190.91) .. (288,197) ;
\draw    (229.2,176.6) .. controls (242.18,180.91) and (253.18,181.91) .. (265.18,171.91) ;
\draw    (265.18,171.91) .. controls (270.18,161.91) and (273.18,163.91) .. (274.18,145.91) ;
\draw    (288,197) .. controls (298.18,189.91) and (306.18,185.91) .. (318.18,184.91) ;
\draw    (274.18,145.91) .. controls (289.18,147.91) and (294.18,143.91) .. (301.18,134.91) ;
\draw    (318.18,184.91) .. controls (314.18,175.91) and (315.18,168.91) .. (323.18,155.91) ;
\draw    (301.18,134.91) .. controls (303.18,142.91) and (312.18,151.91) .. (323.18,155.91) ;
\draw    (214.18,154.91) .. controls (215.18,162.91) and (219.18,169.91) .. (229.2,176.6) ;
\draw    (186.18,148.91) .. controls (194.18,153.91) and (200.18,157.91) .. (214.18,154.91) ;
\draw    (177.2,199.6) .. controls (180.18,191.91) and (180.18,185.91) .. (174.18,178.91) ;
\draw    (174.18,178.91) .. controls (182.18,174.91) and (190.18,163.91) .. (186.18,148.91) ;
\draw    (173.18,130.91) .. controls (172.18,141.91) and (181.18,146.91) .. (186.18,148.91) ;
\draw    (147.18,132.91) .. controls (156.18,134.91) and (161.18,134.91) .. (173.18,130.91) ;
\draw    (152.18,170.91) .. controls (164.18,172.91) and (167.18,173.91) .. (174.18,178.91) ;
\draw    (139.18,147.91) .. controls (147.18,151.91) and (153.18,159.91) .. (152.18,170.91) ;
\draw    (139.18,147.91) .. controls (145.18,145.91) and (148.18,138.91) .. (147.18,132.91) ;
\draw    (252.18,127.91) .. controls (262.18,132.91) and (269.18,137.91) .. (274.18,145.91) ;
\draw    (307.18,103.91) .. controls (296.18,115.91) and (299.18,126.91) .. (301.18,134.91) ;
\draw    (252.18,127.91) .. controls (259.18,118.91) and (259.18,111.91) .. (257.18,99.91) ;
\draw    (257.18,99.91) .. controls (272.18,98.91) and (280.18,94.91) .. (289.18,87.91) ;
\draw    (289.18,87.91) .. controls (293.18,94.91) and (300.18,102.91) .. (307.18,103.91) ;
\draw    (173.18,130.91) .. controls (172.18,121.91) and (175.18,113.91) .. (183.18,110.91) ;
\draw    (147.18,132.91) .. controls (147.18,124.91) and (144.18,118.91) .. (142.18,109.91) ;
\draw    (142.18,109.91) .. controls (148.18,105.91) and (152.18,99.91) .. (156.18,92.91) ;
\draw    (183.18,110.91) .. controls (181.18,106.91) and (178.18,101.91) .. (180.18,95.91) ;
\draw    (156.18,92.91) .. controls (164.18,95.91) and (168.18,98.91) .. (180.18,95.91) ;
\draw    (137.13,124.59) -- (147.18,132.91) ;
\draw    (129.13,148.59) -- (139.18,147.91) ;
\draw    (124.13,126.59) -- (137.13,124.59) ;
\draw    (120.13,139.59) .. controls (126.4,141.9) and (126.4,143.9) .. (129.13,148.59) ;
\draw    (120.13,139.59) .. controls (124.4,135.9) and (127.4,131.9) .. (124.13,126.59) ;
\draw    (285.2,59.86) .. controls (285.2,69.86) and (284.2,77.86) .. (289.18,87.91) ;
\draw    (242.68,80.43) .. controls (249.68,87.43) and (253.2,84.86) .. (257.18,99.91) ;
\draw    (265.68,53.43) .. controls (271.68,60.43) and (276.68,59.43) .. (285.2,59.86) ;
\draw    (242.68,80.43) .. controls (247.85,76.39) and (247.85,69.39) .. (247.85,61.39) ;
\draw    (247.85,61.39) .. controls (254.85,61.39) and (259.85,60.39) .. (265.68,53.43) ;
\draw    (289.18,87.91) .. controls (298.21,83.79) and (303.21,78.79) .. (305.21,70.79) ;
\draw    (307.18,103.91) .. controls (319.35,95.31) and (316.35,97.31) .. (328.35,93.31) ;
\draw    (330.73,77.59) .. controls (326.73,84.59) and (326.73,88.59) .. (328.35,93.31) ;
\draw    (305.21,70.79) .. controls (310.53,75.33) and (317.53,73.33) .. (323.53,69.33) ;
\draw    (323.53,69.33) .. controls (323.27,71.93) and (325.27,74.93) .. (330.73,77.59) ;
\draw  [dash pattern={on 4.5pt off 4.5pt}]  (239.18,233.4) .. controls (221.08,203.36) and (205.52,197.94) .. (185.3,179.15) .. controls (165.08,160.36) and (151.08,136.36) .. (102.77,127.61) ;
\draw  [dash pattern={on 4.5pt off 4.5pt}]  (126.73,82.65) .. controls (164.73,118.65) and (170.73,153.65) .. (196.73,168.65) .. controls (222.73,183.65) and (218.51,180.31) .. (225.73,187.65) .. controls (232.95,194.98) and (252.31,211.72) .. (256.37,232.78) ;
\draw  [line width=3] [line join = round][line cap = round] (144.53,142.33) .. controls (145,142.33) and (145.53,142.86) .. (145.53,143.33) ;
\draw  [line width=3] [line join = round][line cap = round] (164.53,133.33) .. controls (164.53,133) and (164.53,132.67) .. (164.53,132.33) ;
\draw  [line width=3] [line join = round][line cap = round] (179.93,174.07) .. controls (180.27,174.07) and (180.6,174.07) .. (180.93,174.07) ;
\draw  [line width=3] [line join = round][line cap = round] (185.93,161.07) .. controls (185.93,161.4) and (185.93,161.73) .. (185.93,162.07) ;
\draw  [line width=3] [line join = round][line cap = round] (213.93,202.07) .. controls (213.93,202.4) and (213.93,202.73) .. (213.93,203.07) ;
\draw  [line width=3] [line join = round][line cap = round] (224.93,188.07) .. controls (224.93,187.73) and (224.93,187.4) .. (224.93,187.07) ;
\draw [color={rgb, 255:red, 208; green, 2; blue, 27 }  ,draw opacity=1 ][line width=2.25]    (227.2,236.6) .. controls (254.43,230.25) and (248.43,233.25) .. (269,233) ;
\draw  [line width=3] [line join = round][line cap = round] (238.73,232.73) .. controls (238.73,233.07) and (238.73,233.4) .. (238.73,233.73) ;
\draw  [line width=3] [line join = round][line cap = round] (254.73,231.73) .. controls (255.2,231.73) and (255.73,232.26) .. (255.73,232.73) ;

\draw (233.18,235.8) node [anchor=north west][inner sep=0.75pt]  [font=\scriptsize]  {$p_{2}$};
\draw (258.37,236.18) node [anchor=north west][inner sep=0.75pt]  [font=\scriptsize]  {$p_{1}$};
\draw (65,119.4) node [anchor=north west][inner sep=0.75pt]  [font=\scriptsize]  {$f^{+}( p_{2})$};
\draw (100,63.4) node [anchor=north west][inner sep=0.75pt]  [font=\scriptsize]  {$f^{+}( p_{1})$};
\draw (230.2,182) node [anchor=north west][inner sep=0.75pt]  [font=\tiny]  {$f_{1}( p_{1})$};
\draw (193.88,155.05) node [anchor=north west][inner sep=0.75pt]  [font=\tiny,rotate=-22.27]  {$f_{2}( p_{1})$};
\draw (165.55,106.97) node [anchor=north west][inner sep=0.75pt]  [font=\tiny,rotate=-66.28]  {$f_{3}( p_{1})$};
\draw (252,191.4) node [anchor=north west][inner sep=0.75pt]  [font=\footnotesize,color={rgb, 255:red, 208; green, 2; blue, 27 }  ,opacity=1 ]  {$H_{j}$};
\draw (237,217.4) node [anchor=north west][inner sep=0.75pt]  [font=\small,color={rgb, 255:red, 208; green, 2; blue, 27 }  ,opacity=1 ]  {$e_{i}$};
\draw (390,160.4) node [anchor=north west][inner sep=0.75pt]    {$\mathbb{D} $};

\end{tikzpicture}

    \caption{In constructing $\phi(\gamma)$, the order of intersection points on $e_i$ matches with the order of their forwarding endpoins on the boundary of $\mathbb{D}$.}
    \label{fig: universal-hexagon}
\end{figure}
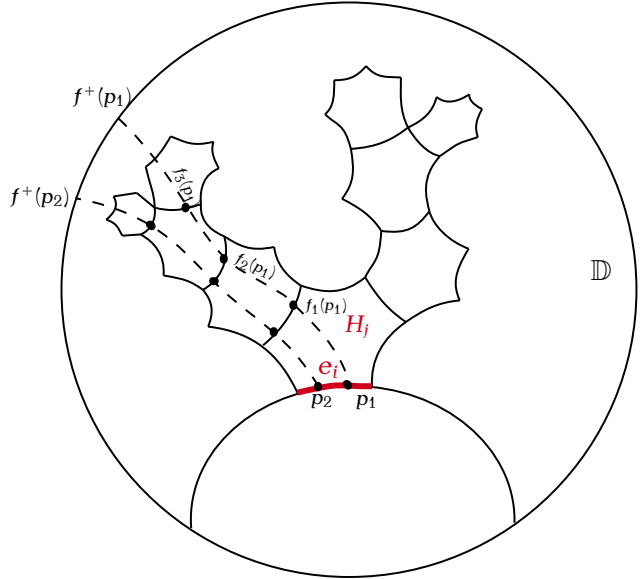

We can also determine the proper ordering using symbolic words. 

Assume that the symbolic word associated to $\gamma$ is
\[
\omega(\gamma)=\omega_1\omega_2\cdots\omega_m,
\]
where $m=\ell_{\mathrm{com}}(\gamma)$. 

We determine the relative location of intersection points on edge $\omega_k$ as follows. 

For each intersection point $p=\gamma\cap\omega_k$, define
\[
f_t(p):=\gamma\cap\omega_{k+t},
\]
where the index $k+t$ is taken modulo $m$. Now consider two outgoing arcs from $p_1,p_2$ on $e_i$ that remain on the same edge for exactly $k$ successive symbols. 
Equivalently, assume that $f_t(p_1)$ and $f_t(p_2)$ lie on the same edge for all $t=0,\dots,k$, but that $f_{k+1}(p_1)$ and $f_{k+1}(p_2)$ lie on distinct edges of the same hexagon (see Figure~\ref{fig: word-k}). If the edge containing $f_{k+1}(p_1)$ lies to the right (to the left) of the edge containing $f_{k+1}(p_2)$ with respect to the cyclic order, excluding the edge containing $f_k(p_1)$ and $f_k(p_2)$, then we set $p_1$ to lie on the right (to the left) of $p_2$. 
This convention defines the proper ordering and relative locations of $p_1$ and $p_2$ on $e_i$ (see Figure~\ref{fig: word-k}).

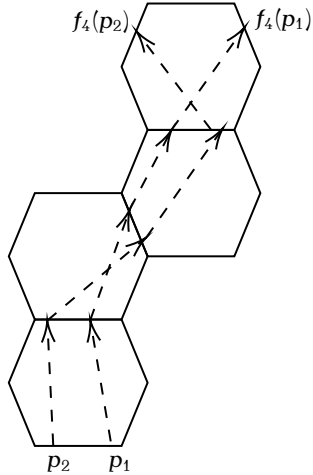
\begin{figure}[h]
    \centering

\tikzset{every picture/.style={line width=0.75pt}} 

\begin{tikzpicture}[x=0.75pt,y=0.75pt,yscale=-1,xscale=1]

\draw   (271,79.9) -- (284.13,48) -- (327.88,48) -- (341,79.9) -- (327.88,111.8) -- (284.13,111.8) -- cycle ;
\draw   (271,143.7) -- (284.13,111.8) -- (327.88,111.8) -- (341,143.7) -- (327.88,175.6) -- (284.13,175.6) -- cycle ;
\draw   (214.13,175.6) -- (227.25,143.7) -- (271,143.7) -- (284.13,175.6) -- (271,207.5) -- (227.25,207.5) -- cycle ;
\draw   (214.13,239.4) -- (227.25,207.5) -- (271,207.5) -- (284.13,239.4) -- (271,271.3) -- (227.25,271.3) -- cycle ;
\draw  [dash pattern={on 4.5pt off 4.5pt}]  (236.53,271.02) -- (233.63,210.01) ;
\draw [shift={(233.53,208.02)}, rotate = 87.27] [color={rgb, 255:red, 0; green, 0; blue, 0 }  ][line width=0.75]    (10.93,-3.29) .. controls (6.95,-1.4) and (3.31,-0.3) .. (0,0) .. controls (3.31,0.3) and (6.95,1.4) .. (10.93,3.29)   ;
\draw  [dash pattern={on 4.5pt off 4.5pt}]  (233.53,208.02) -- (282,167.3) ;
\draw [shift={(283.53,166.02)}, rotate = 139.97] [color={rgb, 255:red, 0; green, 0; blue, 0 }  ][line width=0.75]    (10.93,-3.29) .. controls (6.95,-1.4) and (3.31,-0.3) .. (0,0) .. controls (3.31,0.3) and (6.95,1.4) .. (10.93,3.29)   ;
\draw  [dash pattern={on 4.5pt off 4.5pt}]  (283.53,166.02) -- (321.05,112.44) ;
\draw [shift={(322.2,110.8)}, rotate = 125] [color={rgb, 255:red, 0; green, 0; blue, 0 }  ][line width=0.75]    (10.93,-3.29) .. controls (6.95,-1.4) and (3.31,-0.3) .. (0,0) .. controls (3.31,0.3) and (6.95,1.4) .. (10.93,3.29)   ;
\draw  [dash pattern={on 4.5pt off 4.5pt}]  (316.2,111.8) -- (279.74,63.61) ;
\draw [shift={(278.53,62.02)}, rotate = 52.89] [color={rgb, 255:red, 0; green, 0; blue, 0 }  ][line width=0.75]    (10.93,-3.29) .. controls (6.95,-1.4) and (3.31,-0.3) .. (0,0) .. controls (3.31,0.3) and (6.95,1.4) .. (10.93,3.29)   ;
\draw  [dash pattern={on 4.5pt off 4.5pt}]  (265.53,269.02) -- (255.53,209.77) ;
\draw [shift={(255.2,207.8)}, rotate = 80.42] [color={rgb, 255:red, 0; green, 0; blue, 0 }  ][line width=0.75]    (10.93,-3.29) .. controls (6.95,-1.4) and (3.31,-0.3) .. (0,0) .. controls (3.31,0.3) and (6.95,1.4) .. (10.93,3.29)   ;
\draw  [dash pattern={on 4.5pt off 4.5pt}]  (255.2,207.8) -- (274.53,153.68) ;
\draw [shift={(275.2,151.8)}, rotate = 109.65] [color={rgb, 255:red, 0; green, 0; blue, 0 }  ][line width=0.75]    (10.93,-3.29) .. controls (6.95,-1.4) and (3.31,-0.3) .. (0,0) .. controls (3.31,0.3) and (6.95,1.4) .. (10.93,3.29)   ;
\draw  [dash pattern={on 4.5pt off 4.5pt}]  (275.2,151.8) -- (295.27,113.57) ;
\draw [shift={(296.2,111.8)}, rotate = 117.7] [color={rgb, 255:red, 0; green, 0; blue, 0 }  ][line width=0.75]    (10.93,-3.29) .. controls (6.95,-1.4) and (3.31,-0.3) .. (0,0) .. controls (3.31,0.3) and (6.95,1.4) .. (10.93,3.29)   ;
\draw  [dash pattern={on 4.5pt off 4.5pt}]  (296.2,111.8) -- (332.04,61.43) ;
\draw [shift={(333.2,59.8)}, rotate = 125.43] [color={rgb, 255:red, 0; green, 0; blue, 0 }  ][line width=0.75]    (10.93,-3.29) .. controls (6.95,-1.4) and (3.31,-0.3) .. (0,0) .. controls (3.31,0.3) and (6.95,1.4) .. (10.93,3.29)   ;

\draw (263,275) node [anchor=north west][inner sep=0.75pt]  [font=\footnotesize]  {$p_{1}$};
\draw (232,275) node [anchor=north west][inner sep=0.75pt]  [font=\footnotesize]  {$p_{2}$};
\draw (245,49.4) node [anchor=north west][inner sep=0.75pt]  [font=\footnotesize]  {$f_{4}( p_{2})$};
\draw (337,48.4) node [anchor=north west][inner sep=0.75pt]  [font=\footnotesize]  {$f_{4}( p_{1})$};

\end{tikzpicture}

    \caption{We determine the relative positions of points on each edge according to their forward path.
}
    \label{fig: word-k}
\end{figure}

It is straightforward to verify that the induced pairwise ordering on $e_i$ is globally compatible among all outgoing intersection points on $e_i$. 
More precisely, consider three intersection points $p_1,p_2,p_3 \subset e_i$. 
If proper ordering requires that $p_1$ lie to the right of $p_2$ and $p_2$ lie to the right of $p_3$, then we can see that proper ordering also requires that $p_1$ lie to the right of $p_3$. 
Thus, the induced relation is transitive and defines a well-defined global ordering of intersection points on $e_i$. 
Consequently, the \textit{proper ordering} and curve $\phi(\gamma)$ are well-defined.

It is easy to see $\phi(\gamma)$ does not have any intersection of type $2(a)$. We now show that $\phi(\gamma)$ also does not have intersections of type $3$. 
Consider two arcs
$$
a_1 = p_1 f_1(p_1), \qquad b_1 = p_2 f_1(p_2),
$$
which start and end on the same pair of edges in $H_j$. Assume that $p_1$ lies to the right of $p_2$ on $e_i$, in the proper ordering. By construction, this implies that there exists an integer $k \ge 1$ such that
$f_i(p_1)$ and $f_i(p_2)$ lie on the same edge for all $i=0,\dots,k$, while the edge containing $f_{k+1}(p_1)$ is on the right of the edge containing $f_{k+1}(p_2)$. By the definition of the proper ordering in $\phi(\gamma)$, $f_1(p_1)$ must lie to the right of $f_1(p_2)$ as well. Consequently, the arcs $a_1$ and $b_1$ do not intersect. This shows that no intersections of type~$3$ occur in $\phi(\gamma)$.

Moreover, the number of intersection points of type~$2$ in $\phi(\gamma)$ is not greater than the number of such intersections in $\gamma$. 
Indeed, we show that each intersection point of type~$2(b)$ in $\phi(\gamma)$ corresponds to at least one intersection point in $\gamma$, and we note that $\phi(\gamma)$ has no intersection points of type~$2(a)$. Consider intersecting arcs
$$
f_{-1}(p_1)p_1 \quad \text{and} \quad f_{-1}(p_2)p_2
$$
in $H_j$ of type $2(b)$. In other words, $p_1$ and $p_2$ lie on the same edge, while the starting points $f_{-1}(p_1)$ and $f_{-1}(p_2)$ lie on distinct edges. 
Assume that $p_1$ lies to the right of $p_2$. By the definition of proper ordering, there exists an integer $k \ge 0$ such that
$f_i(p_1)$ and $f_i(p_2)$ lie on the same edge for all $i=0,\dots,k$, while the edge containing $f_{k+1}(p_1)$ lies to the right of the edge containing $f_{k+1}(p_2)$. Let $\omega(p)$ denote the edge containing $p$ in the hexagon that includes the arc $pf_1(p)$. 
Recall that any two linearly homotopic curves admit the same symbolic coding. It is easy to see that any closed curve whose symbolic word contains the subwords
$$
\omega(f_{-1}(p_1))\,\omega(p_1)\,\omega(f_1(p_1))\,\cdots\,\omega(f_{k+1}(p_1))
\quad \text{and} \quad
\omega(f_{-1}(p_2))\,\omega(p_2)\,\omega(f_1(p_2))\,\cdots\,\omega(f_{k+1}(p_2))
$$
must have an intersection between the corresponding subarcs, since their endpoints are linked. 
In particular, since $\gamma$ is linearly homotopic to $\phi(\gamma)$ and hence has the same symbolic word, $\gamma$ also contains such an intersection. Therefore, the total number of intersection points of type $2$ does not increase under the map $\phi$.

\end{proof}

Curve $\phi(\gamma)$ that satisfies the properties of Theorem \ref{thm: linear.homotop} is called a \emph{linearly modified representative} of $\gamma$.

\paragraph{Counting words.} We apply the following combinatorial results in the proof of Theorem \ref{thm: P-closed-geod}.

Given a linearly modified curve $\gamma$, we associate a symbolic word to each edge $e_i$ of a hexagon $H_j$ using the outgoing arcs of $\gamma$ from $e_i$, as follows. 
Let $p_1,\dots,p_m$ be the intersection points of outgoing arcs from $e_i$, ordered from left to right. 
We assign to $e_i$ the word $\omega_1 \omega_2 \cdots \omega_m,$ where $\omega_k$ denotes the edge of $H_j$ containing the point $f_1(p_k)$. Similarly, using the incoming arcs of $\gamma$ to $e_i$, we assign another symbolic word to the edge $e_i$ by recording the edges containing the corresponding backward points. In total, a closed curve $\gamma$ determines $48g-48$ such words.

\begin{Lem}\label{lem: match.words}
Let $\gamma$ be a linearly modified closed curve. 
If we have all $48g-48$ assigned words to the edges for the incoming and outgoing arcs of $\gamma$ then we can uniquely determine the symbolic word $\omega(\gamma)$. 
\end{Lem}
\begin{proof}

Let $e$ and $e'$ be two edges of a hexagon.
Let $\omega=\omega_1\cdots\omega_m$ denote the word associated to the outgoing arcs from $e$, and let $\omega'=\omega'_1\cdots\omega'_{m'}$ denote the word associated to the incoming arcs to $e'$.
Let $p_1,\ldots,p_m$ and $p'_1,\ldots,p'_{m'}$ be the corresponding intersection points on $e$ and $e'$, respectively.

There is an arc connecting $p_s$ on $e$ to $p'_t$ on $e'$ if and only if $\omega_s=e'$ and $\omega'_t=e$.
Thus, the pair $(\omega,\omega')$ determines the set of endpoints of all arcs connecting $e$ to $e'$.

It remains to specify how these endpoints are paired.
Since $\gamma$ has no intersections of type $3$, there is a unique way to do so: the $i$-th point on $e$ must be connected to the $i$-th point on $e'$.
Once all arcs between edges of hexagons are connected in this manner, the symbolic coding $\omega(\gamma)$ is uniquely determined.

\end{proof}

Assume that $\epsilon>0$ is small enough constant compared to $n$.

\begin{Lem}\label{lem: number.permutation.hexagon}
Consider a hexagon with edges $e_1,\dots,e_6$, and points $p_1,\dots,p_n$ on edge $e_1$. 
For each $k=1,\dots,n$, connect $p_k$ to the midpoint of one of the edges $e_2,\dots,e_6$. This choice determines a word $\omega = \omega_1 \cdots \omega_n$, where $\omega_k = e_j$ if $p_k$ connects to the midpoint of $e_j$.

We say that a word $\omega$ is $(n,\epsilon)$-\emph{admissible} if the corresponding collection of arcs has at most $\epsilon n^2$ intersection points inside the hexagon. 
The number of admissible words is bounded above by 
$$
\leq 16 \epsilon^2 n^4 (\frac{e}{\sqrt{\epsilon}})^{8\sqrt{\epsilon}n}.
$$

\end{Lem}

\begin{proof}

First, consider a triangle $e_1e_2e_3$ instead of a hexagon and connect arcs to the midpoints of two edges $e_2,e_3$. Assume that $e_1,e_2,e_3$ are clockwise. 

The number of admissible words that has $\leq \sqrt{\epsilon}n$ arcs connected to $e_2$ is 
$$
\leq \sqrt{\epsilon}n {n \choose \sqrt{\epsilon}n}.  
$$ 
Indeed, the coefficient $\sqrt{\epsilon}n$ represents the possible number of arcs connected to $e_2$ which can be $1,\dots,\lfloor \sqrt{\epsilon}n \rfloor$. Moreover, there are ${n \choose k}$ words with exactly $k$ arcs connected to $e_2$ and ${n \choose k} \leq {n \choose \sqrt{\epsilon}n}$ when $\epsilon$ is small compared to $n$ and $k\leq \sqrt{\epsilon}n$.

If the number of points connected to $e_2$ is $\geq \sqrt{\epsilon}n$ then choose the $\sqrt{\epsilon}n$ rightmost points connected to $e_2$; there are ${n \choose \sqrt{\epsilon}n}$ possible ways to indicate these points on $e_1$. Among these $\sqrt{\epsilon}n$ points let $q$ be the leftmost one. See Figure \ref{fig: triangle}. The points on the right of $q$ are completely determined where to attach; $\sqrt{\epsilon}n$ of them were chosen to connect to $e_2$, and the rest must be attached to $e_3$.

Now consider the points on the left of $q$. At most $\sqrt{\epsilon}n$ of them can be attached to $e_3$ unless we have $>\epsilon n^2$ intersection points (between these arcs and the $\sqrt{\epsilon}n$ chosen arcs connected to $e_2$). Therefore, we have $\leq \sqrt{\epsilon}n$ points connected to $e_3$ on the left of $q$. We conclude that the number of ways to determine the subword for the left points is
$$
\leq \sqrt{\epsilon}n {n \choose \sqrt{\epsilon}n}.
$$

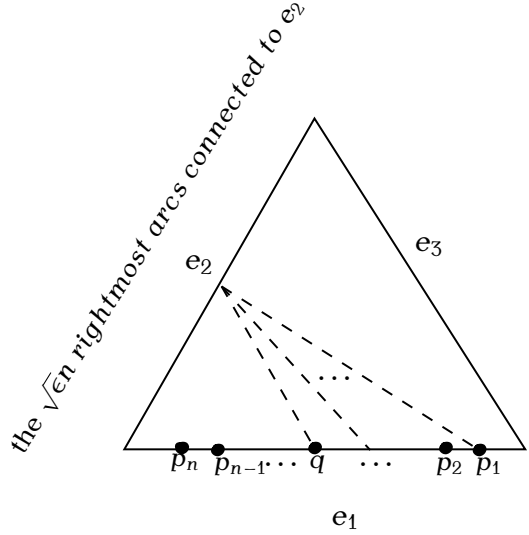
\begin{figure}[h]
    \centering

\tikzset{every picture/.style={line width=0.75pt}} 

\begin{tikzpicture}[x=0.75pt,y=0.75pt,yscale=-1,xscale=1]

\draw   (258.9,48.8) -- (365.2,215.8) -- (163,215.8) -- cycle ;
\draw  [line width=3] [line join = round][line cap = round] (190.51,214.21) .. controls (190.51,214.68) and (191.04,215.21) .. (191.51,215.21) ;
\draw  [line width=3] [line join = round][line cap = round] (191.51,214.21) .. controls (194.07,214.21) and (188.7,214.81) .. (191.51,216.21) .. controls (192.34,216.62) and (193.18,215.86) .. (193.51,215.21) .. controls (194.23,213.76) and (192.4,213.76) .. (191.51,214.21) .. controls (191.21,214.36) and (191.51,214.87) .. (191.51,215.21) ;
\draw  [line width=3] [line join = round][line cap = round] (209.51,215.21) .. controls (212.07,215.21) and (206.7,215.81) .. (209.51,217.21) .. controls (210.34,217.62) and (211.18,216.86) .. (211.51,216.21) .. controls (212.23,214.76) and (210.4,214.76) .. (209.51,215.21) .. controls (209.21,215.36) and (209.51,215.87) .. (209.51,216.21) ;
\draw  [line width=3] [line join = round][line cap = round] (258.51,214.21) .. controls (261.07,214.21) and (255.7,214.81) .. (258.51,216.21) .. controls (259.34,216.62) and (260.18,215.86) .. (260.51,215.21) .. controls (261.23,213.76) and (259.4,213.76) .. (258.51,214.21) .. controls (258.21,214.36) and (258.51,214.87) .. (258.51,215.21) ;
\draw  [line width=3] [line join = round][line cap = round] (341.51,215.21) .. controls (344.07,215.21) and (338.7,215.81) .. (341.51,217.21) .. controls (342.34,217.62) and (343.18,216.86) .. (343.51,216.21) .. controls (344.23,214.76) and (342.4,214.76) .. (341.51,215.21) .. controls (341.21,215.36) and (341.51,215.87) .. (341.51,216.21) ;
\draw  [line width=3] [line join = round][line cap = round] (324.51,214.21) .. controls (327.07,214.21) and (321.7,214.81) .. (324.51,216.21) .. controls (325.34,216.62) and (326.18,215.86) .. (326.51,215.21) .. controls (327.23,213.76) and (325.4,213.76) .. (324.51,214.21) .. controls (324.21,214.36) and (324.51,214.87) .. (324.51,215.21) ;
\draw  [dash pattern={on 4.5pt off 4.5pt}]  (212.05,133.58) -- (342,216.2) ;
\draw  [dash pattern={on 4.5pt off 4.5pt}]  (212.05,133.58) -- (287,216.2) ;
\draw  [dash pattern={on 4.5pt off 4.5pt}]  (212.05,133.58) -- (260,217.2) ;

\draw (339,218.6) node [anchor=north west][inner sep=0.75pt]  [font=\small]  {$p_{1}$};
\draw (319,218.6) node [anchor=north west][inner sep=0.75pt]  [font=\small]  {$p_{2}$};
\draw (255,217.6) node [anchor=north west][inner sep=0.75pt]  [font=\small]  {$q$};
\draw (207,218.6) node [anchor=north west][inner sep=0.75pt]  [font=\small]  {$p_{n-1}$};
\draw (185,216.6) node [anchor=north west][inner sep=0.75pt]  [font=\small]  {$p_{n}$};
\draw (280,220.6) node [anchor=north west][inner sep=0.75pt]    {$\dotsc $};
\draw (232,220.6) node [anchor=north west][inner sep=0.75pt]    {$\dotsc $};
\draw (259,177.6) node [anchor=north west][inner sep=0.75pt]    {$\dotsc $};
\draw (100.88,213.21) node [anchor=north west][inner sep=0.75pt]  [font=\small,rotate=-301.77,xslant=0]  {the $\sqrt{\epsilon } n\ rightmost\ arcs\ connected\ to\ e_{2}$};
\draw (266,245.6) node [anchor=north west][inner sep=0.75pt]    {$e_{1}$};
\draw (192,116.6) node [anchor=north west][inner sep=0.75pt]    {$e_{2}$};
\draw (308,107.6) node [anchor=north west][inner sep=0.75pt]    {$e_{3}$};

\end{tikzpicture}

    \caption{Point $q$ divides the points on $e_1$ into two blocks}
    \label{fig: triangle}
\end{figure}

Therefore, in total, the number of such possible words is 

$$
\leq \sqrt{\epsilon}n {n \choose \sqrt{\epsilon}n}+ \sqrt{\epsilon}n {n \choose \sqrt{\epsilon}n}^2 \leq 2\sqrt{\epsilon}n {n \choose \sqrt{\epsilon}n}^2 \leq 2\sqrt{\epsilon}n (\frac{e}{\sqrt{\epsilon}})^{2\sqrt{\epsilon}n}
$$

The last inequality follows from Stirling's approximation for ${n \choose \sqrt{\epsilon}n}$.

By iterating this procedure four times, we obtain the desired result for a hexagon. 
Equivalently, we first count the number of words formed using two symbols, $e_2$ and $b$, where the symbol $b$ means we have an arc from $e_1$ to one of the remaining four edges $e_3, e_4,e_5,e_6$. Once such words are fixed, determining how the positions labeled by $b$ are assigned to the specific edges $e_3,e_4,e_5,e_6$ becomes an analogous problem.

Consequently, the total number of admissible words associated to $e_1$ is
$$
\leq (2\sqrt{\epsilon}n (\frac{e}{\sqrt{\epsilon}})^{2\sqrt{\epsilon}n})^4.
$$ 
\end{proof}

Let $\mathcal{G}$ be the set of all closed geodesics on $X$.
Recall that: 
$$
P_{\epsilon}(T):=| \{\gamma \in \mathcal{G}: \ell_X(\gamma)\leq T, \, i(\gamma,\gamma)\leq \epsilon T^2 \}|.
$$

\textbf{Proof of Theorem \ref{thm: P-closed-geod}.}  From Lemma \ref{lemma: com.length} we have

$$
P_{\epsilon}(T) \leq |\{ \gamma \in \mathcal{G}| \, \ell_{com}(\gamma) \leq c(X)T, i(\gamma,\gamma) \leq \epsilon T^2 \}|.
$$

Now instead of counting closed geodesics $\gamma$ we aim to count the number of possible words $\omega(\phi(\gamma))$, where $\phi(\gamma)$ is the linearly modified representatives of $\gamma$ as described in Theorem \ref{thm: linear.homotop}.

For $\gamma \in P_{\epsilon}(T)$, by Theorem \ref{thm: linear.homotop}, $\phi(\gamma)$ has combinatorial length $\leq c(X)T$ and the number of its intersection points of type $2$ is $\leq \epsilon T^2$. On one hand, from Lemma \ref{lem: match.words},  we have
$$
|\{ \gamma \in \mathcal{G}| \, \ell_{com}(\gamma) \leq c(X)T, i(\gamma,\gamma) \leq \epsilon T^2 \}| \leq A(n,\delta)^{48g-48}, 
$$
where $A(n, \delta)$ is the number of possible admissible words, as described in Lemma \ref{lem: number.permutation.hexagon}, for $n=c(X)T$ and $\delta=\epsilon/c(X)^2$ instead of $\epsilon$ appearing in the statement of that lemma. Therefore, we have
$$
2P_{\epsilon}(T)\leq (2\sqrt{\epsilon}T(\frac{ec(X)}{\sqrt{\epsilon}})^{2\sqrt{\epsilon}T})^{4(48g-48)}\leq (\frac{ec(X)}{\sqrt{\epsilon}})^{385\sqrt{\epsilon}T(g-1)},
$$
as required. We conssidered $2P_{\epsilon}(T)$ instead of $P_{\epsilon}(T)$ since the counted curves are oriented.
\qed

\begin{Remark}\label{remark: b}
     The constant $b_X$ appearing in Theorem \ref{thm: P-closed-geod} is equal to $e\cdot c(X)$ where $c(X)$ is the constant appearing in Lemma \ref{lemma: com.length}. 
\end{Remark}

\begin{Remark}\label{remark: bc.bound}
    We have the following bounds on the interaction strength and the Bers' constant: 
    $$
    I(X)\leq \frac{4}{\sys(X)^2}\, ,  \,  \, \, \, \, \,  L_g\leq 26(g-1)\, ; \, \, \, \, \, \, 
    $$
    see Proposition $2.4$ of \cite{int-sys-T}, and Theorem $5.1.2$ of \cite{Bsr}. These bounds imply
   $$
   c(X)\leq \frac{2184(g-1)^2}{\sys(X)^2}+\frac{12}{\sys(X)}
   $$
   and since $\sys(X)\leq 2\log(4g-2)\leq 4(g-1)$ (see \cite[Lemma 5.2.1]{Bsr}) we have
   $$
   b_X \leq ec(X)\leq e\big( \frac{2184(g-1)^2}{\sys(X)^2}+ \frac{48(g-1)}{\sys(X)^2}\big)\leq \frac{5974(g-1)^2}{\sys(X)^2}, 
   $$
     where $b_X$ is the variable appearing in Theorem \ref{thm: P-closed-geod}. 
\end{Remark}

The symbolic coding has also been employed in other settings to study the geometry of curves. See \cite{chas-self-marcov}\cite{series-symbolic}\cite{katok-symbolic}. For example, in joint work with Zhang~\cite{Tina-zhang-halo}, symbolic codings based on train tracks are used to analyze limit sets arising from measured laminations. In that paper, the setting and goals are different from the present work, but both approaches reflect the usefulness of symbolic descriptions in understanding entropy, dynamical phenomena, and self-intersection number. 

\section{Measure-theoretic entropy}\label{sec: entropy}

In this section, we prove Theorem \ref{thm: entropy<P} and Theorem \ref{thm: main.entropy}.

\paragraph{Entropy.} Let $M$ be a compact metric space with distance function $d$ and a continuous flow $f=\{f_t\}_{t \in \mathbb{R}}$. Assume that $\mu$ is a Borel probability measure on $M$,  which is $f$-invariant.
Let $h_{M}(\mu)$ be the measure-theoretic entropy of $\mu$. When $\mu$ is ergodic with respect to $f$, Katok proved the following equivalent definition for $h_{M}(\mu)$.     

For $T>0$ and $v_1,v_2 \in M$, define 
$$
d_T(v_1,v_2):= \max \limits_{0 \leq t \leq T} d(f_t(v_1), f_t(v_2)).
$$

Given $T,\epsilon >0,\, \delta \in (0,1)$, define $N_{\mu}(T, \epsilon, \delta)$ as the minimum number of balls of radius $\epsilon$, with respect to metric $d_T$, which is required to cover a subset of $M$ with measure $\geq 1-\delta$. 

\begin{Prop}\label{prop: katok} (Katok) Assume that $\mu$ is an ergodic invariant measure.  For every $\delta \in (0,1)$ we have 
$$
h_{M}(\mu)= \lim \limits_{\epsilon \rightarrow 0} \liminf \limits_{T \rightarrow \infty} \frac{\log N_{\mu}(T, \epsilon, \delta)}{T}=
\lim \limits_{\epsilon \rightarrow 0} \limsup \limits_{T \rightarrow \infty} \frac{\log N_{\mu}(T, \epsilon, \delta)}{T}.
$$
\end{Prop}

In our case, $M=T_1(X)$, the unit tangent bundle of $X$, $f$ is the geodesic flow, $d$ is the metric on $T_1(X)$ induced by the hyperbolic metric of $X$, and $\mu$ is a geodesic current. For simplicity, we write $h_X(\mu)$ instead of $h_{T_1(X)}(\mu)$.

\paragraph{Remark.} Given a geodesic current $C \in \mathcal{C}_g$ with $\ell_X(C)=1$, the entropy $h_X(C)$ takes values in the interval $[0,1]$. In particular, $h_X(C)=0$ when $C$ is a closed geodesic or a measured lamination \cite{BirSer} and $h_X(C)=1$ only when $C=L_X$, the Liouville measure.

\paragraph{Anosov closing lemma.}
 Let $M$ be a negatively curved manifold, and $T_1(M)$ the unit tangent bundle of $M$. Let $f_t$ and $d$ denote the geodesic flow and metric on $T_1(M)$, respectively. Anosov closing lemma states:
 
 \begin{Lem}\label{lem: anosov}
For any $\epsilon>0$ there exists $\delta>0$ such that if 
$$
d(v, f_t(v))< \delta, \, \, \, \, \, \, v \in T_1(M)
$$   
then there is a geodesic $\gamma=\{ f_t(w)\}_{s=0}^{t'}$ of length $t'$ where $|t-t'|<\epsilon$ and 
$$
 d(f_s(v), f_s(w)) <\epsilon \quad \text{for} \enspace all \enspace  0\leq s \leq t .
$$
 \end{Lem}

In Lemma \ref{lem: anosov}, $\gamma$ is the geodesic representative of the closed curve obtained from connecting endpoints $v$ and $f_t(v)$ of $\{f_s(v)\}_{s=0}^{t}$. 

See \cite{katokBook} for a reference to this lemma. 

 \paragraph{Self-intersection number of random geodesics.} 
 Recall that 
 
 $$
P_{\epsilon}(T):=| \{\gamma \in \mathcal{G}: \ell_X(\gamma)\leq T, \, i(\gamma,\gamma)\leq \epsilon T^2 \}|.
$$

Define
$$
P_{\epsilon}:= \lim \limits_{T \to \infty} \frac{\log{P_{\epsilon}(T)}}{T}.
$$

It is a well-known result of Katok that $h_X(\mu) \leq P_{\infty}$ for any invariant measure $\mu$ \cite[Thm. 2.1]{Ktok.entrpy}. He related $N_{\mu}(T,\epsilon,\delta)$ to the number of closed geodesics with length $\leq (1+\epsilon')T$, for some $\epsilon'>0$ depending on $\epsilon$. We show that we can also add a constraint on the self-intersection number of these closed geodesics. More precisely, we show that the self-intersection number of these closed geodesics is about $i(\mu,\mu)T^2$ when $\mu$ is ergodic. Then we conclude Theorem \ref{thm: entropy<P}. We first record the following results, which will be used in the proof.

\begin{Lem}\label{lem: p-cont}
   The function $P_\epsilon$ is continuous for $\epsilon>0$.
    \end{Lem}
    \begin{proof}
        On one hand, we know that $P_{\epsilon}$ is increasing. On the other hand, for any $\delta>0$ we have:
        $$
        P_{\epsilon+\delta} \leq \sqrt{\frac{\epsilon+\delta}{\epsilon}} P_{\epsilon},
        $$
        since 
        $$
        P_{\epsilon+\delta}(T) \subset P_{\epsilon}(\sqrt{\frac{\epsilon+\delta}{\epsilon}}T). 
        $$
    \end{proof}
    
\begin{Lem}\label{lem: hyp}
Let $abcd$ be a hyperbolic $4-$gon with right angles at $b$ and $c$. Assume that the length of the edges $ab$ and $cd$ are $T$, and the length of $bc$ is $\delta$. See Figure \ref{fig: 4gon}. Then we have
$$
\cosh(L)=\cosh(T)^2(\cosh(\delta)-1)+1,
$$
where $L$ is the length of $ad$.
\begin{figure}[h]
    \centering

\tikzset{every picture/.style={line width=0.75pt}} 

\begin{tikzpicture}[x=0.75pt,y=0.75pt,yscale=-0.8,xscale=0.8]

\draw    (210,195) .. controls (241.2,164.8) and (287.2,170.8) .. (310,195) ;
\draw    (152.2,88.8) .. controls (183.2,114.8) and (199.2,147.8) .. (210,195) ;
\draw    (310,195) .. controls (314.2,159.8) and (318.2,110.8) .. (347.2,74.8) ;
\draw [color={rgb, 255:red, 95; green, 159; blue, 31 }  ,draw opacity=1 ]   (152.2,88.8) .. controls (190.2,47.8) and (295.2,41.8) .. (347.2,74.8) ;
\draw  [color={rgb, 255:red, 0; green, 0; blue, 0 }  ,draw opacity=1 ][line width=0.75] [line join = round][line cap = round] (208.27,187.93) .. controls (210.16,187.93) and (211.56,183.93) .. (213.27,183.93) .. controls (214.37,183.93) and (215.27,189.15) .. (215.27,189.93) ;
\draw  [color={rgb, 255:red, 0; green, 0; blue, 0 }  ,draw opacity=1 ][line width=0.75] [line join = round][line cap = round] (304.67,190.27) .. controls (304.67,188.24) and (305.67,186.29) .. (305.67,184.27) .. controls (305.67,182.32) and (308.72,187.27) .. (310.67,187.27) ;

\draw (137,76.4) node [anchor=north west][inner sep=0.75pt]    {$a$};
\draw (212,198.4) node [anchor=north west][inner sep=0.75pt]    {$b$};
\draw (312,198.4) node [anchor=north west][inner sep=0.75pt]    {$c$};
\draw (352,67.4) node [anchor=north west][inner sep=0.75pt]    {$d$};
\draw (239,33.4) node [anchor=north west][inner sep=0.75pt]    {$L$};
\draw (168,127.4) node [anchor=north west][inner sep=0.75pt]    {$T$};
\draw (328,122.4) node [anchor=north west][inner sep=0.75pt]    {$T$};
\draw (253,154.4) node [anchor=north west][inner sep=0.75pt]    {$\delta $};

\end{tikzpicture}

    \caption{A hyperbolic $4$-gon}
    \label{fig: 4gon}
\end{figure}
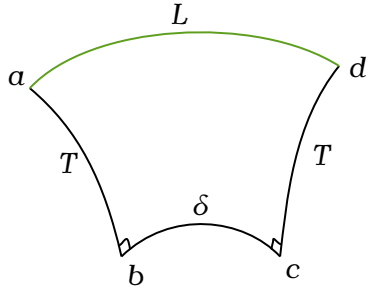
\end{Lem}
\begin{proof}
    The proof is implied from Equation $(2,3,2)$ in \cite{Bsr}.
\end{proof}
Let $f_{[0,t]}(v)$ be the geodesic arc of length $t$ in $T_1(X)$ starting at $v \in T_1(X)$.

\begin{Prop}\label{prop: entropy.a.e.self-int}
Let $c_0>0$ be a sufficiently small constant depending on $X$. Let $C$ be an ergodic geodesic current. Then for $C$-almost every $v \in T_1(X)$, and for any sequence $t_n \to \infty$ satisfying
$$
d\bigl(f_{t_n}(v),v\bigr) \leq c_0,
$$
we have
$$
\lim_{n \to \infty} \frac{i(\gamma_n,\gamma_n)}{\ell_X(\gamma_n)^2} = i(C,C),
$$
where $\gamma_n$ denotes the geodesic representative of the closed curve obtained from connecting the endpoints $v$ and $f_{t_n}(v)$ of the geodesic segment $f_{[0,t_n]}(v)$.
\end{Prop}
\begin{proof}
By \textit{Birkhoff's ergodic theorem}, for $C$-almost every $v \in T_1(X)$ we have
\begin{equation}\label{erg}
\lim_{T \to \infty} \frac{1}{T} \int_0^T g(f_t(v))\,dt
=
\int_{T_1(X)} g \, dC
\end{equation}
for any $g \in C^1(T_1(X))$.

On the other hand, the geodesic representative $\gamma_n$ is a long closed geodesic that shadows the segment $f_{[0,t_n]}(v)$. More precisely, by the Anosov closing lemma (Lemma ~\ref{lem: anosov}), the curve $\gamma_n$ and $f_{[0,t_n]}(v)$ remain within $\epsilon_0$-neighborhood of each other, and $|\ell_X(\gamma_n) - t_n| \le \epsilon_0$.

For simplicity, set $T := t_n$. Assume that function $g$ is bounded above by $b>0$. We aim to show $\gamma_n/\ell_X(\gamma_n) \to C$ as $n \to \infty$. We have
$$
\left|
\frac{1}{T} \int_0^T g(f_t(v))\,dt
-
\frac{1}{\ell_X(\gamma_n)} \int_0^{\ell_X(\gamma_n)} g(\gamma_n(t))\,dt
\right| \le
\frac{1}{T} \int_0^T |g(f_t(v)) - g(\gamma_n(t))|\,dt
+ b\,\epsilon_0.
$$

We split the integral as
\begin{equation}\label{equ: integrals}
\le
\frac{1}{T} \int_{\sqrt{T}}^{T-\sqrt{T}} |g(f_t(v)) - g(\gamma_n(t))|\,dt
+
\frac{1}{T} \int_{I_T} |g(f_t(v)) - g(\gamma_n(t))|\,dt
+ b\,\epsilon_0,
\end{equation}
where $I_T = [0,T] \setminus [\sqrt{T},\,T-\sqrt{T}]$.

The corresponding points on $\gamma_n$ and $f_{[0,T]}(v)$ are $\epsilon_0$-close. By basic hyperbolic geometry, these two geodesics achieve their minimum distances at two interior points, and the distance between corresponding points increases monotonically away from this region. Let $D$ be the maximum distance between the corresponding points on $\gamma_n$ and $f_{[\sqrt{T},T-\sqrt{T}]}(v)$. 

By applying Lemma~\ref{lem: hyp}, we have

\[
\cosh(D)
\le
1 + \frac{2\cosh(T)^2(\cosh(\delta)-1)}{e^{2\sqrt{T}}},
\qquad
1 + \cosh(T)^2(\cosh(\delta)-1) \leq \cosh(\epsilon_0).
\]

In particular, it is straightforward to see that $D \to 0$ as $T \to \infty$. Since function $g$ is uniformly continuous (being continuous on a compact space), the first integral in Equation \eqref{equ: integrals} converges to zero. The remaining terms are
\[
\frac{1}{T} \int_{I_T} |g(f_t(v)) - g(\gamma_n(t))|\,dt + b\,\epsilon_0
\le
\frac{4\sqrt{T}\,b}{T} + b\,\epsilon_0,
\]
which also tends to zero as $T \to \infty$.

Therefore,
\[
\lim_{n \to \infty}
\frac{1}{\ell_X(\gamma_n)}
\int_0^{\ell_X(\gamma_n)} g(\gamma_n(t))\,dt
=
\int_{T_1(X)} g \, dC.
\]

This shows that the normalized closed geodesics $\gamma_n / \ell_X(\gamma_n)$ converge weakly to $C$. By continuity of the intersection form $i(\cdot,\cdot)$, it follows that
\[
\frac{i(\gamma_n,\gamma_n)}{\ell_X(\gamma_n)^2}
\to
i(C,C)
\quad \text{as } n \to \infty,
\]
as claimed.
\end{proof}

Let $c_0$ be a sufficiently small constant depending on $X$. Define a subset $A_{(T,\epsilon)} \subset T_1(X)$ that consists of vectors $v$ such that for some $t\in[T, T(1+\epsilon)]$ with $d(v,f_t(v)) \le c_0$ we have
\[
\left|
\frac{i(\gamma,\gamma)}{\ell_X(\gamma)^2}
-
i(C,C)
\right|
\le \epsilon
\]
where $\gamma$ is the closed geodesic obtained from $f_{[0,t]}(v)$ by connecting its endpoints.

Let $B_{(T, \epsilon)}:= \{ v \in T_1(X) \, : \,  \, \exists \,  t \in[T, (1+ \epsilon)T] \, \, \, \, that \, \, \, \, d(v, f_t(v))< \epsilon   \}$.

\begin{Lem} \label{lem: set=1} For any $\epsilon > 0$, as $T \rightarrow \infty$ we have:

\begin{enumerate}
\item $ C(B_{(T, \epsilon)}) \to 1$,
\item $C(A_{(T, \epsilon)}) \to 1$.
  
\end{enumerate}
\end{Lem}
\begin{proof}

For part $1$, cover $T_1(X)$ by finitely many balls of radius $\epsilon/2$. Let $K$ be one of these balls and $\chi_K$ the characteristic function of $K$. By ergodicity of $C$ for $a.e.$ $v \in K \subset T_1(X)$ we have:
\begin{equation}\label{erg.ch}
\lim \limits_{T \to \infty} \frac{1}{T} \int_0^T \chi_K(f_t(v)) dt=  \int_{T_1(X)} \chi_K(v) dC= C(K).
\end{equation}

Therefore, for large enough $T$, we have:
$$
\int_0^T \chi_K(f_t(v)) dt < \int_0^{T(1+\epsilon)} \chi_K(f_t(v)) dt.
$$
We conclude that for $C-a.e.$ $v \in K$ and larg enough $T$ there exists $t \in [T, T(1+\epsilon)]$ such that $f_t(v)$ is in $K$ and, therefore, $d(v, f_t(v))<\epsilon$.
Applying this to all balls in the covering of $T_1(X)$ implies part $1$. Part $2$ is implied from part $1$ and Proposition \ref{prop: entropy.a.e.self-int}.
\end{proof}

\textbf{Proof of Theorem \ref{thm: entropy<P}.} Recall that $N_C(T, \epsilon, \delta)$ is the minimum number of $\epsilon-$balls with respect to $d_T$ metric such that they cover a set of measure $> 1- \delta$ in $T_1(X)$. From Proposition \ref{prop: katok}, we know that $h_X(C)$ is equal to the exponential growth rate of $N_C(T, \epsilon, 1/2)$. Define $Y := A_{(T,\delta)} \subset T_1(X)$. By Lemma \ref{lem: set=1}, for every $ \delta > 0$ and $T$ large enough $C(Y) \geq \frac{1}{2}$. We will choose $\delta$ later in terms of $\epsilon$. 

We say that a set $\Lambda$ is an$(T, \epsilon)-$separated set if the $d_T$ distance between every two points of $\Lambda$ is $>\epsilon$. Let $\Lambda_{(T,\epsilon)}$ be a maximal $(T, 3\epsilon)-$separated subset of $Y$. Then balls of radius $3\epsilon$, with respect to metric $d_T$, with centers at the points of $\Lambda_{(T,3\epsilon)}$ cover $Y$. Therefore, we have:
\begin{equation}
 |\Lambda_{(T,\epsilon)}| \geq N_C(T, 3\epsilon, \frac{1}{2}).
\end{equation}
Given $v \in \Lambda_{(T,\epsilon)}$, using the fact that $v \in A_{(T,\delta)}$ and applying Anosov's closing lemma, we obtain a closed geodesic $\gamma_v$ which is close to $f_{[0,T]}(v)$. More precisely, choose $\delta>0$ small enough so that $\gamma_v$ is $\epsilon-$close to $f_{[0,t]}(v)$, for some $t\in [T,(1+\epsilon)T]$, and $\ell_X(\gamma_v)<(1+ \epsilon)T$.
 From Anosov closing lemma, we know that there exists a point $q_v$ on $\gamma_v$ such that $d_T(q_v, v)< \epsilon$. 
When $v_1, v_2\in \Lambda_{(T,\epsilon)}$ are distinct, we have:
\begin{equation}\label{dis}       
d_T(q_{v_1}, q_{v_2}) > d_T(v_1,v_2)-d_T(v_1, q_{v_1})-d_T(v_2, q_{v_2})> 3\epsilon-2\epsilon=\epsilon
\end{equation}

 Equation ~\ref{dis} implies that the number of points $q_{v'}$ on each obtained closed geodesic $\gamma_v$ is $<\ell_X(\gamma)/\epsilon$. We know $\ell_X(\gamma)<(1+ \epsilon)T$. Therefore, the map from $\Lambda_{(\epsilon, T)}$ to these closed geodesics is at most $(1+\epsilon)T/\epsilon$ to $1$. We conclude that the number of obtained closed geodesics is

\begin{equation}
 \geq \frac{ |\Lambda_{(\epsilon, T)}|}{(1+ \epsilon)T/ \epsilon} \geq \frac{\epsilon}{T(1+ \epsilon)}N_C(T, 3\epsilon, \frac{1}{2}).
\end{equation}

Note that $|i(\gamma_v,\gamma_v)/\ell_X(\gamma_v)^2-i(C,C)| \leq \epsilon$, therefore, the self-intersection number of $\gamma_v$ is

$$
i(\gamma_v,\gamma_v) \leq (i(C,C)+\epsilon)\ell_X(\gamma_v)^2,
$$
and its length is $\leq (1+\epsilon)T$.

This implies that
$$
P_{i(C, C)+\epsilon}((1+\epsilon)T) \geq \frac{\epsilon}{T(1+ \epsilon)}N_C(T, 3\epsilon, \frac{1}{2}).
$$
Let $T'=(1+\epsilon)T$, we have

$$
\frac{\log{P_{i(C, C)+\epsilon}(T')}}{T'} \geq \frac{1}{(1+\epsilon)} \frac{ \log \frac{\epsilon}{T(1+ \epsilon)}+ \log N_C(T, 3\epsilon, \frac{1}{2})}{T};
$$
tending $T \rightarrow \infty$ we have
$$
p_{i(C, C)+\epsilon} \geq \frac{1}{1+\epsilon} \lim \limits_{T \to \infty} \frac{\log N_C(T,3\epsilon,\frac{1}{2})}{T}.
$$
Now by letting $\epsilon\to 0$, from Lemma \ref{lem: p-cont} and Proposition \ref{prop: katok}, we have $P_{i(C,C)}\geq h_X(C)$.
\qed
\paragraph{Proof of Theorem \ref{thm: main.entropy}.} From Theorems \ref{thm: entropy<P} and \ref{thm: P-closed-geod}, we have:

$$
h_X(C)\leq P_{i(C,C)}\leq 385(g-1)\sqrt{i(C,C)}|\log(\frac{b_X}{\sqrt{i(C,C)}})|,
$$
as required.
\qed

\begin{Remark}\label{remark: b1.1,3} The constants $b_X$ appearing in Theorem \ref{thm: main.entropy} and Theorem \ref{thm: P-closed-geod} are the same. See Remark \ref{remark: bc.bound}.
\end{Remark}
\newpage

\bibliographystyle{math}
\bibliography{biblio}

\end{document}